\documentclass[10.5pt,a4paper]{amsart}

\usepackage{lipsum}% http://ctan.org/pkg/lipsum
\makeatletter
\g@addto@macro{\endabstract}{\@setabstract}
	\newcommand{\authorfootnotes}{\renewcommand\thefootnote{\@fnsymbol\c@footnote}}%
\makeatother

\usepackage{comment}
\usepackage{amssymb}
\usepackage{mathtools}
\usepackage{mathrsfs}
\usepackage{mathptmx}
\usepackage[utf8]{inputenc}
\usepackage{bm}
\usepackage{bbm}
\usepackage{amsmath}
\usepackage{amsfonts}
\makeatletter
\def\amsbb{\use@mathgroup \M@U \symAMSb}
\makeatother

\usepackage{comment}
\usepackage{graphicx}
\usepackage{grffile}
\usepackage{caption}
\usepackage{subcaption}

\usepackage{bbold}

\usepackage{mathptmx}

\newcommand{\supp}{\textrm{supp}\,}

\newcommand{\bga}{\begin{aligned}}
\newcommand{\ena}{\end{aligned}}

\newcommand{\bge}{\begin{enumerate}}
\newcommand{\ene}{\end{enumerate}}

\newcommand{\red}[1]{{\color{red} #1}}
\newcommand{\blue}[1]{{\color{blue} #1}}
\newcommand{\magenta}[1]{{\color{magenta} #1}}
\newcommand{\hide}[1]{}

 \usepackage{dutchcal}
\usepackage{fancyhdr}
\usepackage{subcaption}

\usepackage{newtxtext,newtxmath}
\usepackage{pgfplots}
\pgfplotsset{compat=1.15}

\usepackage{xcolor}
\definecolor{webgreen}{rgb}{0,.5,0}
\definecolor{webbrown}{rgb}{.6,0,0}
\definecolor{RoyalBlue}{cmyk}{1, 0.50, 0, 0}
\usepackage[colorlinks=true, breaklinks=true, urlcolor=webbrown, linkcolor=RoyalBlue, citecolor=webgreen,backref=page]{hyperref}

\usepackage{tikz}
\usepackage{pict2e}
\usepackage{todonotes}
\usetikzlibrary{decorations.markings}

\DeclareSymbolFont{bbold}{U}{bbold}{m}{n}
\DeclareSymbolFontAlphabet{\mathbbold}{bbold}

\newcommand{\R}{{\mathbb R}}
\newcommand{\C}{{\mathbb C}}

\newcommand{\Z}{{\mathbb Z}}

\newcommand{\N}{{\mathbb N}}

\newcommand{\al}{\alpha}

\newcommand{\ga}{\gamma}
\newcommand{\Ga}{\Gamma}

\newcommand{\ep}{\varepsilon}

\newcommand{\om}{\omega}
\newcommand{\Om}{\Omega}

\newcommand{\di}{\displaystyle}

\newcommand{\ovl}{\overline}

\newcommand{\ii}{\textrm{i}}
\newcommand{\dd}{\textrm{d}}

\newcommand{\qandq}{\quad \text{and} \quad}

	\newtheorem{theorem}{Theorem}
	
	\newtheorem{definition}[theorem]{Definition}
	
	\newtheorem{remark}[theorem]{Remark}
	\newtheorem{lemma}[theorem]{Lemma}

\usepackage[T1]{fontenc}

\usepackage{mathtools}

\makeatletter
\DeclareRobustCommand\widecheck[1]{{\mathpalette\@widecheck{#1}}}
\def\@widecheck#1#2{%
	\setbox\z@\hbox{\m@th$#1#2$}%
	\setbox\tw@\hbox{\m@th$#1%
		\widehat{%
			\vrule\@width\z@\@height\ht\z@
			\vrule\@height\z@\@width\wd\z@}$}%
	\dp\tw@-\ht\z@
	\@tempdima\ht\z@ \advance\@tempdima2\ht\tw@ \divide\@tempdima\thr@@
	\setbox\tw@\hbox{%
		\raise\@tempdima\hbox{\scalebox{1}[-1]{\lower\@tempdima\box
				\tw@}}}%
	{\ooalign{\box\tw@ \cr \box\z@}}}
\makeatother

\newcommand{\fbseries}{\unskip\setBold\aftergroup\unsetBold\aftergroup\ignorespaces}
\makeatletter
\newcommand{\setBoldness}[1]{\def\fake@bold{#1}}
\makeatother

\begin{document}

\tikzset{middlearrow/.style={
			decoration={markings,
				mark= at position 0.6 with {\arrow{#1}} ,
			},
			postaction={decorate}
		}
	}\tikzset{middlearrow/.style={
			decoration={markings,
				mark= at position 0.6 with {\arrow{#1}} ,
			},
			postaction={decorate}
		}
	}

	\tikzset{->-/.style={decoration={
				markings,
				mark=at position #1 with {\arrow{latex}}},postaction={decorate}}}
	
	\tikzset{-<-/.style={decoration={
				markings,
				mark=at position #1 with {\arrowreversed{latex}}},postaction={decorate}}}

\title{ OPENNESS OF REGULAR REGIMES OF COMPLEX RANDOM MATRIX MODELS}

\maketitle

\begin{comment}
	\authorfootnotes
	Marco Bertola\footnote{Marco.Bertola@concordia.ca, Marco.Bertola@sissa.it}\textsuperscript{2,4,5}, Pavel Bleher\footnote{pbleher@iupui.edu}\textsuperscript{3},
	Roozbeh Gharakhloo\footnote{roozbeh.gharakhloo@colostate.edu}\textsuperscript{1}, and Kenneth T-R  McLaughlin\footnote{kenmcl@rams.colostate.edu}\textsuperscript{1}  \bigskip

	\textsuperscript{1}Department of Mathematics, Colorado State University, Fort Collins, CO 80521, USA\par
	\textsuperscript{2}Department of Mathematics and Statistics, Concordia University
	1455 de Maisonneuve W., Montr´eal, Qu´ebec, Canada H3G 1M8 \par
	\textsuperscript{3} Department of Mathematical Sciences, Indiana University-Purdue University Indianapolis, 402 N. Blackford St., Indianapolis, IN 46202
	\textsuperscript{4}SISSA, International School for Advanced Studies, via Bonomea 265, Trieste, Italy\par
	\textsuperscript{5}Centre de recherches math´ematiques, Universit´e de Montr´eal
	C. P. 6128, succ. centre ville, Montr´eal, Qu´ebec, Canada H3C 3J7.\par
	\bigskip
\end{comment}

\begin{center}
    \authorfootnotes
  Marco Bertola\footnote{Department of Mathematics and Statistics, Concordia University
1455 de Maisonneuve W., Montr{\' e}al, Qu{\' e}bec, Canada H3G 1M8,  E-mail: Marco.Bertola@concordia.ca,

Centre de recherches math{\' e}matiques, Universit{\' e} de Montr{\' e}al, C. P. 6128, succ. centre ville, Montr{\' e}al, Qu{\' e}bec, Canada H3C 3J7. Email: bertola@crm.umontreal.ca
%\smallskip

 SISSA, International School for Advanced Studies, via Bonomea 265, Trieste, Italy.  E-mail: Marco.Bertola@sissa.it,
%\smallskip
},
  Pavel Bleher\footnote{Department of Mathematical Sciences, Indiana University-Purdue University Indianapolis, 402 N. Blackford St., Indianapolis, IN 46202, Blackford St., Indianapolis, IN 46202, USA. e-mail: pbleher@iupui.edu},
  Roozbeh Gharakhloo\footnote{Department of Mathematics, Colorado State University, Fort Collins, CO 80521, USA, E-mail: roozbeh.gharakhloo@colostate.edu}, Kenneth T-R  McLaughlin\footnote{Department of Mathematics, Colorado State University, Fort Collins, CO 80521, USA, E-mail: kenmcl@rams.colostate.edu},
  Alexander Tovbis\footnote{Department of Mathematics, University of Central Florida, 4000 Central Florida Blvd., Orlando, FL 32816-1364, USA E-mail: alexander.tovbis@ucf.edu}
  \par \bigskip
\end{center}

\begin{center}
\footnotesize	\textit{Dedicated to the memory of Harold Widom}
\end{center}

\begin{abstract}
 Consider the general complex polynomial external field
\begin{equation*}
V(z)=\frac{z^{k}}{k}+\sum_{j=1}^{k-1} \frac{t_j z^j}{j}, \qquad t_j \in \C, \quad k \in \N. 
\end{equation*}
Fix an equivalence class $\mathcal{T}$ of admissible contours whose members approach $\infty$ in two different directions and consider the associated max-min energy problem \cite{KS}. When $k=2p$, $p \in \N$, and $\mathcal{T}$ contains the real axis, we show that the set of parameters $t_1, \cdots, t_{2p-1}$ which gives rise to a regular $q$-cut max-min (equilibrium) measure, $1 \leq q \leq 2p-1 $,  is an open set in $\C^{2p-1}$. We use the implicit function theorem to prove that the endpoint equations are solvable in a small enough neighborhood of a regular $q$-cut point. We also establish the real-analyticity of the real and imaginary parts of the end-points for all $q$-cut regimes, $1 \leq q \leq 2p-1$, with respect to the real and imaginary parts of the complex parameters in the external field. Our choice of even $k$ and the equivalence class $\mathcal{T} \ni \R$ of admissible contours is only for the simplicity of exposition and our proof extends to all possible choices in an analogous way.

\end{abstract}

\tableofcontents

\section{Introduction and Main Results}

\begin{comment}

\begin{enumerate}
	\item Real Case. Properties of Eq measure for OPRL. Also recall the main  results of \cite{KuijMcL}. 
	\item Complex Case. Choice of OP contours. Motivations.
	\item Existence of curves with $S$-property (Rakhmanov, Stahl, Nuttal, Martinez-Finklestein, Kuijlaars-Silva)
	\subitem No proof that the measure is continuous w.r.t. the parameters.
	\item Quadratic Differentials
	\item Definition of the regular q-cut case
	\item State the problem and the main result. Explain why do we consider external field of even degree?
	\item Discussion (including the results of Bertola-Tovbis)
\end{enumerate}
	content...
\end{comment}

The present paper is part of an ongoing project whose main objective is the investigation of the phase diagram and phases of the unitary ensemble of random matrices 
with a general complex potential
\begin{equation}\label{int1}
V(z;\boldsymbol{t})=\frac{z^{2p}}{2p}+\sum_{j=1}^{2p-1} \frac{t_j z^j}{j}, \qquad t_j \in \C, \quad p \in \N, 
\end{equation}
in the complex space  of the vector of the  parameters \[
\boldsymbol{t}=(t_1, \cdots, t_{2p-1})\in\C^{2p-1}.
\]
The unitary ensemble under consideration is defined as the complex measure on the space  of $n \times n$ Hermitian random matrices,
\begin{equation} \label{int2}
 \frac{1}{\Tilde{\mathcal{Z}}_{n}} e^{-n \mathrm{Tr}\,V(M;\boldsymbol{t})} \dd M,
\end{equation}
where 
\begin{equation} \label{int3}
\Tilde{\mathcal{Z}}_{n}(\boldsymbol{t}) = \int_{\mathcal{H}_n} e^{-n \mathrm{Tr}\,V(M;\boldsymbol{t})} \dd M
\end{equation}
is the {\it partition function}. As well known (see, e.g., \cite{BL}), the ensemble of eigenvalues of $M$,
\[
Me_k=z_ke_k, \qquad \;k=1,\ldots,n,
\]
is given by the probability distribution
\begin{equation}\label{int4}
\frac{1}{\mathcal{Z}_{n}(\boldsymbol{t})}
\prod_{1\leq j<k\leq n} (z_j-z_k)^2\prod^{n}_{j=1} \exp\left[-nV(z_j;\boldsymbol{t})\right] \dd z_1 \cdots \dd z_n,
\end{equation}
where 
\begin{equation}\label{int5}
\mathcal{Z}_{n}(\boldsymbol{t}) = \int^{\infty}_{-\infty} \cdots \int^{\infty}_{-\infty} \prod_{1\leq j<k\leq n} (z_j-z_k)^2\prod^{n}_{j=1} \exp\left[-nV(z_j;\boldsymbol{t})\right] \dd z_1 \cdots \dd z_n,
\end{equation}
is the {\it eigenvalue partition function}. The partition functions 
$\mathcal {Z}_{n}$ and 
$\Tilde{\mathcal{Z}}_{n}$ are related by the formula,
\begin{equation}\label{int6}
\frac{\mathcal {Z}_{n}(\boldsymbol{t})}{\Tilde{\mathcal{Z}}_{n}(\boldsymbol{t})} = 
\frac{1}{\pi^{n(n-1)/2}}\prod_{k=1}^nk!.
\end{equation}
Formulae \eqref{int5}, \eqref{int6} are well known for real polynomial potentials $V(z)$ of even degree
(see, e.g., \cite{BL}), and their proof for a complex $V(z)$  goes through without any change.

By Heine's formula (see e.g. \cite{SzegoOP}) the multiple integral in \eqref{int5} is, up to a multiplicative constant, the determinant of the Hankel matrix $H_n[w] := \{ w_{j+k} \}_{k,j=0, \ldots, n}$, where $w(x;\boldsymbol{t})\equiv\exp[-nV(x;\boldsymbol{t})]$ and $w_\ell$ is the $\ell$-th moment of the weight $w(x;\boldsymbol{t})$. Correspondingly, one can also consider the system of monic orthogonal polynomials $\{P_n(z;\boldsymbol{t})\}_{n \in \Z_{\geq 0}}$ satisfying
\begin{eqnarray}\label{OPs}
\int_{\Gamma} P_{n}(z;\boldsymbol{t}) z^{k} w(z;\boldsymbol{t}) dz = 0 \ , \quad \mbox{for } \quad k = 0, 1, \ldots, n-1.
\end{eqnarray} 
The connection of this system of orthogonal polynomials and the partition function \eqref{int5} can be seen as follows: the orthogonal polynomial of degree $n$ exists and is unique if the partition function $\mathcal{Z}_{n}(\boldsymbol{t})$, or the $n \times n$ Hankel determinant $\det H_n[w]$, is nonzero. Indeed, the existence follows from the explicit formula
\begin{equation}
P_n(z;\boldsymbol{t}) \equiv P_n(z) = \frac{1}{\det H_n[w]} \det \begin{pmatrix}
w_0 & w_1 & \cdots & w_{n-1} & w_{n} \\
w_1 & w_{2} & \cdots  & w_{n} & w_{n+1} \\
\vdots & \vdots & \reflectbox{$\ddots$} & \vdots & \vdots \\
w_{n-1} & w_{n} & \cdots  & w_{2n-2} & w_{2n-1} \\
1 & z & \cdots & z^{n-1} & z^n
\end{pmatrix},
\end{equation}
and the uniqueness follows from the fact that the linear system to find the coefficients of $P_n(z;\boldsymbol{t}) \equiv z^n + \sum_{j=0}^{n-1} a_j(\boldsymbol{t}) z^j$, is of the form $H_n[w] \boldsymbol{a} = \boldsymbol{b}$, and thus can be inverted if the Hankel determinant is nonzero.

  It is well known that the normalized counting measure for the zeros of these orthogonal polynomials weakly converges to the associated \textit{equilibrium measure} $\nu_{\rm eq}$ (See e.g. \cite{ZeroDistribution} and references therein). For a review of the definitions and properties of the equilibrium measure in the cases where the external field is real and complex see \S \ref{EqMeasRealCase} and \S \ref{EqMeasCplxCase} .

\begin{comment}

	\blue{Connect A) the eigenvalues of this random matrix ensemble, to B) the zeros of the orthogonal polynomials and C) the support of the equilibrium measure. Then define what the equilibrium measure is in the real case and in the complex case briefly. leave the review of exact details to sections \ref{EqMeasRealCase} and \ref{EqMeasCplxCase}}

\magenta{What do we know about accumulation of eigenvalues of the Hermitian random matrices with the complex external fields? In the real case we know that} 

\[ \lim_{n \to \infty} \mathbb{E}[N_n[(a,b)]] = \int_{a}^{b} h(x)R_+^{1/2}(x) \dd x  \]

\magenta{where $N_n[(a,b)]$ is the number of eigenvalues in the interval $(a,b)$ of a random Hermitian matrix with the distribution }

$$ \frac{1}{\tilde{\mathcal{Z}}_{nN}} e^{-N \mathrm{Tr}\,V(M)} \dd M$$

and 

$$\dd \mu_{\mbox{eq}} \equiv h(x)R_+^{1/2}(x) \dd x$$
\magenta{is the equilibrium measure.}

\red{But when the external field is complex, the matrices are still Hermitian, thus their eigenvalues are real. So here the eigenvalues can not accumulate on the support of the equilibrium measure which is a \textit{complex} Jordan arc?}

\end{comment}

The properties of the equilibrium measure when the external field is real have been studied extensively over the last two decades or so (see e.g. \cite{DKM,KuijMcL,SaffTotik} and references therein), and we briefly review these properties in \S \ref{EqMeasRealCase}. In the real case, the contour of orthogonality for the orthogonal polynomials with respect to $e^{-nV(x;\boldsymbol{t})}$, $\boldsymbol{t}\in \R^{2p-1}$, is the real line and the equilibrium measure is supported on finitely many closed real intervals.  One does not need to deal with the problem of \textit{choosing} the contour of integration for orthogonal polynomials in the case where the external field is real, as the solution of the associated extremal problem for the equilibrium measure automatically ensures that the real line is the correct contour of integration. 

In this work we are considering polynomials defined by a "complex orthogonality condition", of the form \eqref{OPs}.  It is easy to see that the polynomials, when they are uniquely determined by the above orthogonality condition, are independent of the choice of contour, within some equivalence class of contours.  Moreover, for a given weight function $w(z;\boldsymbol{t}) \equiv e^{- n V(z;\boldsymbol{t})}$, there are multiple possible choices of equivalence classes of contours (see for instance \cite{BertolaTovbis2015,BertolaTovbis2016,KS}), and each equivalence class yields a different sequence of orthogonal polynomials. 

Even though for each choice of the equivalence class of contours, our method would work, for the sake of simplicity of exposition, we will restrict ourselves as follows:  
We will assume that the external field is a polynomial of degree 2p (see \eqref{int1}), and we will choose the class of contours of integration that are all in the same equivalence class as the real axis.  

As opposed to the case of a real measure on the real axis defining more classical polynomials all of whose zeros are real, the case of complex orthogonality produces polynomials whose zeros exhibit more complicated behavior. In fact, as the degree of the polynomials tends to infinity, the zeros accumulate on nontrivial curves in the complex plane.

%(though this is only known to be true for orthogonality conditions with analytic weights).

In order to carry out an asymptotic analysis of the orthogonal polynomials with complex weights, a new problem arises which is the effective selection of a contour of integration for which subsequent analysis is possible.  It turns out that the effective selection of the contour of integration determines within it the accumulation set of the zeros of the orthogonal polynomials, which is the support of the equilibrium measure (suitably generalized to the complex case).  

The problem of determining this important set in the plane, which is later used as a portion of the contour of integration, is actually connected to a classical energy problem dating back at least to Gauss - the energy of a continuum of particles in the presence of an external field that experiences a repelling force whose potential is logarithmic.  The set is determined by considering, for each member $\Gamma$ of the class of admissible contours $\mathcal{T}$, the energy minimization problem on $\Gamma$, and then selecting a contour $\Gamma_{0} \in \mathcal{T}$ that maximizes this minimum energy.   In other words, $\Ga_0$ solves the following \textit{max-min} problem: 
\begin{equation}\label{max-min problem}
\underset{\Ga \in \mathcal{T}}{\mathrm{max}} \left\{ \underset{\substack{\supp(\nu) \subset \Ga \\ \nu(\C)=1 }}{\mathrm{min}} \left\{ \underset{\Ga\,\times\,\Ga}{\iint} \log \frac{1}{|z-s|}\,\dd \nu(z) \dd \nu(s) + \int_{\Ga} \Re V(s)\, \dd \nu(s) \right\} \right\}.
\end{equation}  
The admissible sectors (in which the admissible equivalence classes of contours could approach $\infty$) are those in which the requirement
\begin{equation}\label{growth condition}
	\lim_{z\to \infty} \Re V \to +\infty
\end{equation}
 holds, which allows one to associate the \textit{Euler-Lagrange characterization} of the equilibrium measure\cite{SaffTotik} \begin{equation}\label{em711}
\bga
& U^{\nu}(z)+\frac{1}{2}\,\Re V(z)=\ell,\quad z\in \supp \nu,\\
& U^{\nu}(z)+\frac{1}{2}\,\Re V(z)\ge \ell,\quad z\in \Ga\setminus \supp \nu,
\ena
\end{equation}
where 
\begin{equation}\label{em811}
U^{\nu}(z)=\int_{\Ga} \log \frac {1}{|z-s|}\,\dd \nu(s) 
\end{equation}
is the {\it logarithmic potential} of the measure $\nu$ \cite{SaffTotik}.  There is quite a history of reseach centering on this variational problem in approximation theory and potential theory. See, for example \cite{KamRakh, KS, MFRakh1,MFRakh2, Rakh, Stahl, Stahl1} and references therein.  

In \cite{KS} the authors prove the quite general result that for an allowable\footnote{characterized by a notion of \textit{non-crossing partitions} of $\{1,\cdots,N\}$, where $N$ is the number of sectors in which \eqref{growth condition} holds, see \cite{KS}.} equivalence class $\mathcal{T}$ of contours, the solution $\Ga_0$ to the above extremal problem exists, the equilibrium measure and, thus, its support $J$ are unique, and the support $J \subset \Ga_0$ of the equilibrium measure is a finite union of disjoint analytic arcs. Moreover, they show that the support $J$ of the equilibrium measure is part of the critical graph of the quadratic differential $Q(z) \dd z^2$ (see \S \ref{SecQD} for some background on quadratic differentials) where $Q$ is the \textit{polynomial} (see Proposition 3.7 of \cite{KS})
\begin{equation}
	Q(z)=\left(-\om(z)+\frac{V'(z)}{2}\right)^2, 
\end{equation}
in which $\om$ is the \textit{resolvent} of the equilibrium measure
\begin{equation}
	\om(z)=\int_J\frac{\dd \nu_{\rm eq}(x)}{z-x}\,,\quad z\in \C\setminus J.
\end{equation}

Summarizing, we will consider the above max-min variational problem which is associated to the orthogonal polynomials with respect to $e^{-nV(z;\boldsymbol{t})}$, $\boldsymbol{t}\in \C^{2p-1}$, in which the contour $\Ga_{\boldsymbol{t}}$ in the complex $z$-plane, being the solution of the max-min problem, is chosen from the members of the equivalence class of contours $\mathcal{T}$ (defined in \S \ref{EqMeasCplxCase} below - each member being a simply connected curve that tends to $\infty$ in two different directions, in sectors surrounding the positive and negative real axis). For a "generic" choice of $\boldsymbol{t} \in \C^{2p-1}$, the support $J_{\boldsymbol{t}}$ of the equilibrium measure is a finite union of disjoint analytic arcs (which are also referred to as \textit{cuts}), at each endpoint the density of the equilibrium measure vanishes like a square root $\dd \nu_V(s;\boldsymbol{t}) =(2\pi \ii)^{-1} h(s;\boldsymbol{t}) \left(\sqrt{R(s;\boldsymbol{t})}\right)_+ \dd s$, where $h(s;\boldsymbol{t})$ and $R(s;\boldsymbol{t})$ are polynomials in $s$, and $R$ has the property that its only zeros are simple zeros at the endpoints of the cuts. Moreover, for a generic $\boldsymbol{t}$ the zeros of $h(s;\boldsymbol{t})$ do not lie on $J_{\boldsymbol{t}}$ and one can find a complementary set to $J_{\boldsymbol{t}}$ to build the desired infinite contour $\Ga_{\boldsymbol{t}}$ so that the requirement outside the support in \eqref{em711} is satisfied. In fact, for a generic $\boldsymbol{t}$ these complementary contours can all be chosen to satisfy the \textit{strict} inequality in \eqref{em711}, or equivalently chosen so that they all lie in the so-called $\boldsymbol{t}$-stable lands:
\begin{equation}\label{Re eta < 0}
\left\{ z : \Re \eta_q(z;\boldsymbol{t})  <0 \right\},
\end{equation}
where
\begin{equation}\label{eta q}
\eta_q(z;\boldsymbol{t}):=- \int^z_{b_q(\boldsymbol{t})} h(s;\boldsymbol{t})\sqrt{R(s;\boldsymbol{t})} \dd s,
\end{equation}
and $b_q(\boldsymbol{t})$ is the rightmost endpoint.

However, we may expect that the above \textit{regularity} properties do not hold for certain choices of $\boldsymbol{t}$. For example, for some values of $\boldsymbol{t}$ it could happen that

\begin{itemize}
	\item[(a)]   one or more zeros of $h(s;\boldsymbol{t})$ coincide with the endpoints and thus alter the square root vanishing of the density at one or more endpoints,  
	\item[(b)]	one or more zeros of $h(s;\boldsymbol{t})$ may hit the support $J_{\boldsymbol{t}}$ of the equilibrium measure, or 
	\item[(c)] it may not be possible to choose the complementary contours to entirely lie in the $\boldsymbol{t}$-stable lands.
\end{itemize}

Such values of $\boldsymbol{t}$ at  which the aforementioned regularity properties fail, also form \textit{boundaries} in the phase space $\C^{2p-1}$, across which the number of support cuts of the equilibrium measure changes.

 Let us highlight these irregularity properties at non-generic parameter values using the complex quartic external field: $$V(z;\sigma) \equiv \frac{z^4}{4} + \sigma \frac{z^2}{2}, \qquad \sigma \in \C,$$ for which one has (regular) one-cut, two-cut, and three-cut regions in the complex $\sigma$-plane which are denoted by $\mathcal{O}_1, \mathcal{O}_2,$ and $\mathcal{O}_3$ respectively. In \cite{BertolaTovbis2015} the phase diagrams for a variety of choices of integration contours for this model have been presented. In \cite{BGM} the particular case of admissible contours  that approach $\infty$ along the real axis was considered and the phase diagram (as shown in Figure \ref{fig:Phase Diagram}\footnote{Figures \ref{fig:Phase Diagram}, \ref{fig:No Connection transition}, and \ref{fig:stable and barren} are taken from \cite{BGM}.}) was proven.  Using the explicit formulae for the end-points of $J_{\sigma}$ and zeros of $h(z;\sigma)$ in the one-cut case, one can easily find that the non-generic parameter values corresponding to case (a) above are only $\sigma=\pm \ii \sqrt{12}$, for which the points $\mp z_0$ (zeros of $h(z;\sigma)$) coincide with the endpoints $\pm b_1$ \cite{BGM,BertolaTovbis2015}. The non-generic points on the boundaries labeled by $\ga_1$ and $\ga_2$ represent the $\sigma$ values for which the zeros of $h(z;\sigma)$ hit the support of the equilibrium measure (see Figure \ref{fig:No Connection transition}).  Figure \ref{fig:stable and barren} corresponds to item (c) above, in which the regions in light blue represent the $\sigma$-stable lands. Figures  \ref{fig:fertile barren -1 1.6} through \ref{fig:stable barren -1 4} show the contour $\Ga_{\sigma}$ for six choices of parameter $\sigma \in \mathcal{O}_1$, while Figure \ref{fig:stable barren -1.15 4} corresponds to a \textit{non-generic} value of $\sigma \in \ga_3$ (see Figure \ref{fig:Phase Diagram}) where the complementary part $\Ga_{\sigma}\setminus J_{\sigma}$ (the orange dashed line in \ref{fig:stable barren -1.15 4}) can not avoid going through at least one point which does not belong to the $\sigma$-stable lands (see item (c) above). Figure \ref{fig:stable barren -1.35 41} corresponds to $\sigma=-1.35+4\ii$ which is clearly not a one-cut parameter as there is no connection from the endpoint $b_1$ to $\infty$ in the sector originally chosen for the orthgonal polynomials, however, it turns out that it is a regular three-cut parameter \cite{BGM}.
 
 \begin{figure}[h]
 	\centering
 	\includegraphics[scale=0.25]{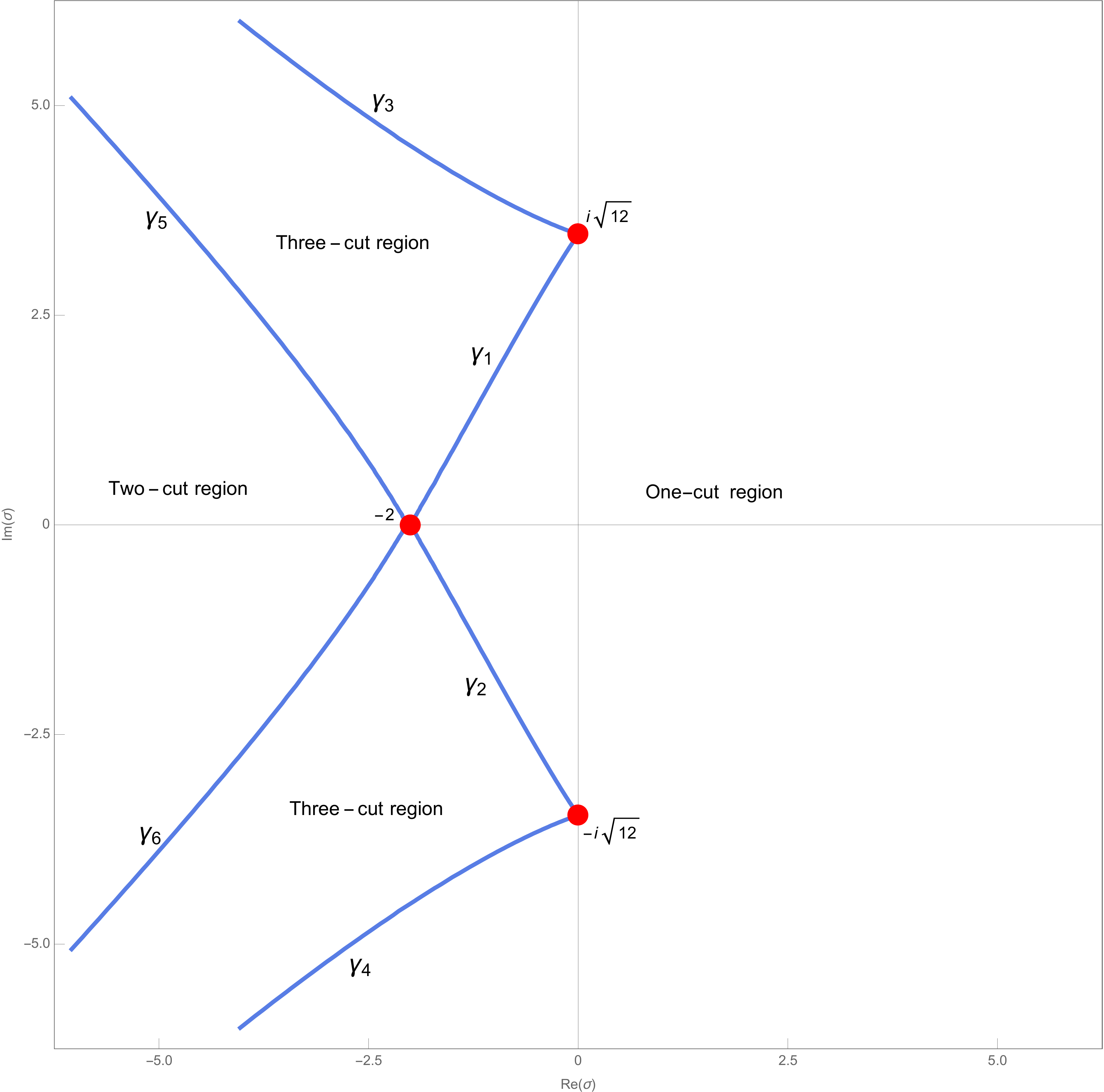}
 	\caption{The phase diagram of the complex quartic random matrix model in the $\sigma$-plane.}
 	\label{fig:Phase Diagram}
 \end{figure}

\begin{figure}[b]
	\centering
	\begin{subfigure}{0.31\textwidth}
		\centering
		\includegraphics[width=\textwidth]{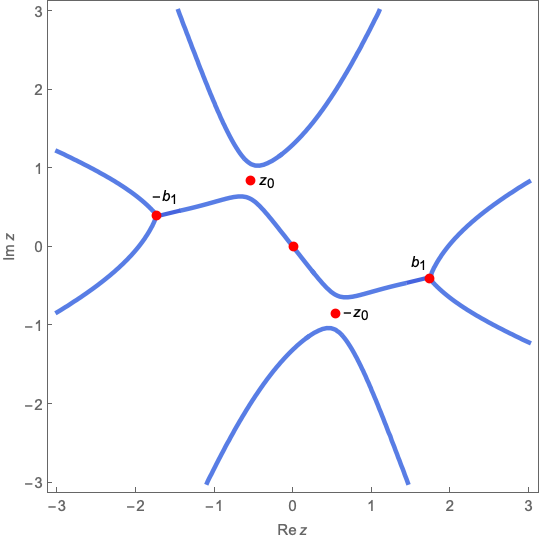}
		\caption{The critical graph $\mathscr{J}^{(1)}_{\sigma}$ of the one-cut quadratic differential for the complex quartic model at a $\sigma \in \mathcal{O}_1$. At this value of $\sigma$ all regularity properties are satisfied.}
		\label{fig:Connection and humps}
	\end{subfigure}%
	\hfill
	\begin{subfigure}{.31\textwidth}
		\centering
		\includegraphics[width=\textwidth]{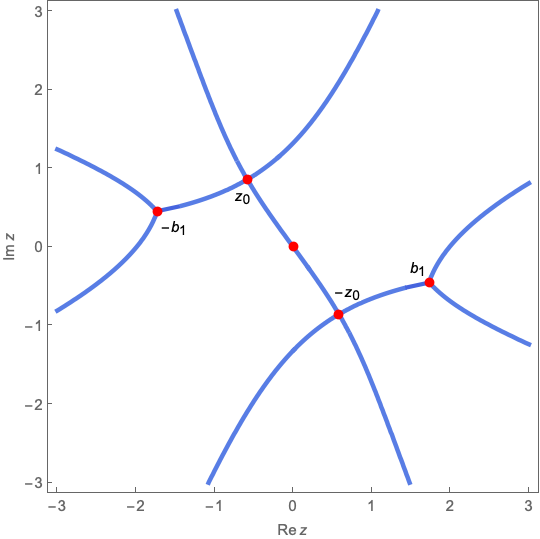}
		\caption{The critical graph $\mathscr{J}^{(1)}_{\sigma}$ at a critical value $\sigma \in \ga_1$ (see Figure \ref{fig:Phase Diagram}). The zeros of $h(z;\sigma)$ at this value hit $J_{\sigma}$, and thus $\sigma \notin \mathcal{O}_1$.}
		\label{fig:Connection about to be lost}    
	\end{subfigure}
	\hfill
	\begin{subfigure}{.31\textwidth}
		\centering
		\includegraphics[width=\textwidth]{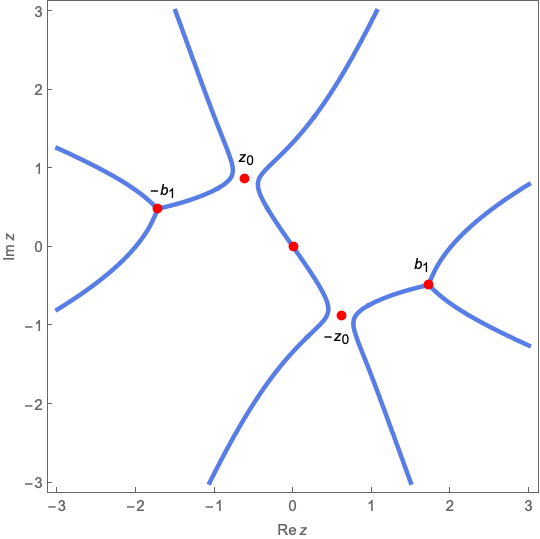}
		\caption{The critical graph $\mathscr{J}^{(1)}_{\sigma}$ at a $\sigma \notin \mathcal{O}_1$. It turns out that this value of $\sigma$ actually is a regular three-cut value as shown in \cite{BGM}.}
		\label{fig:No Connection}    
	\end{subfigure}
	\caption{Snapshots of the continuous deformation (see Theorem \ref{continuous deformations of one cut critical graph1}) of the critical graph $\mathscr{J}^{(1)}_{\sigma}$
% ( the collection of all points $z$ satisfying  $		\Re \left[ \eta_1(z;\sigma) \right]=0$ ) from a $\sigma \in \mathcal{O}^*_1$ where there is a connection from  $-b_1$ to $b_1$, to a $\sigma \notin \mathcal{O}^*_1$ where $-z_0$ and $z_0$ both lie on $\mathscr{J}^{(1)}_{\sigma}$, and finally to a $\sigma \notin \mathcal{O}^*_1$ where there is no connection from  $-b_1$ to $b_1$.
}
	\label{fig:No Connection transition}
\end{figure}

\begin{figure}[h]
	\centering
	\begin{subfigure}{0.24\textwidth}
		\centering
		\includegraphics[width=\textwidth]{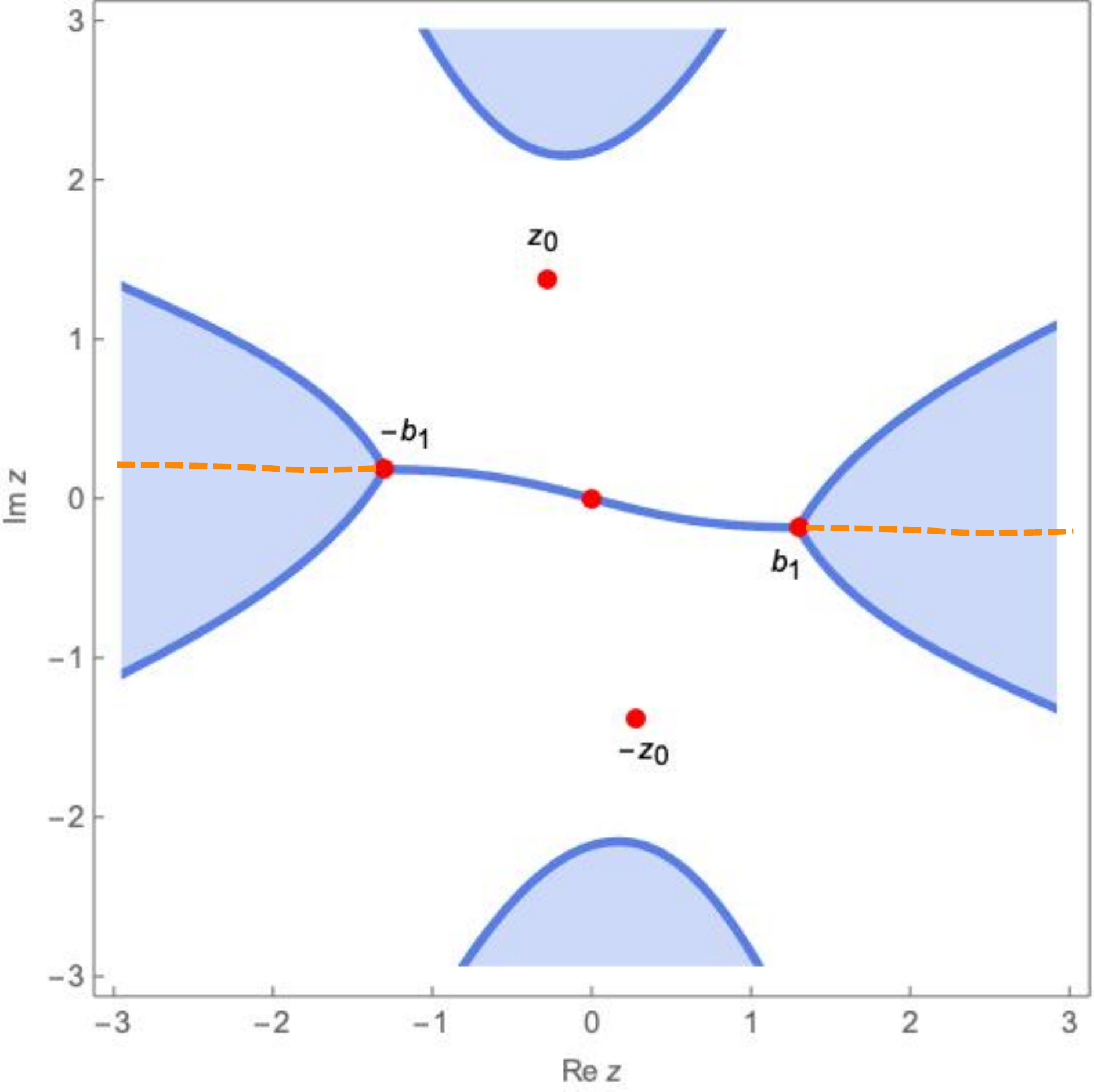}
		\caption{$\sigma=1+\ii \in \mathcal{O}_1$.}
		\label{fig:fertile barren -1 1.6}
	\end{subfigure}%
	\begin{subfigure}{.24\textwidth}
		\centering
		\includegraphics[width=\textwidth]{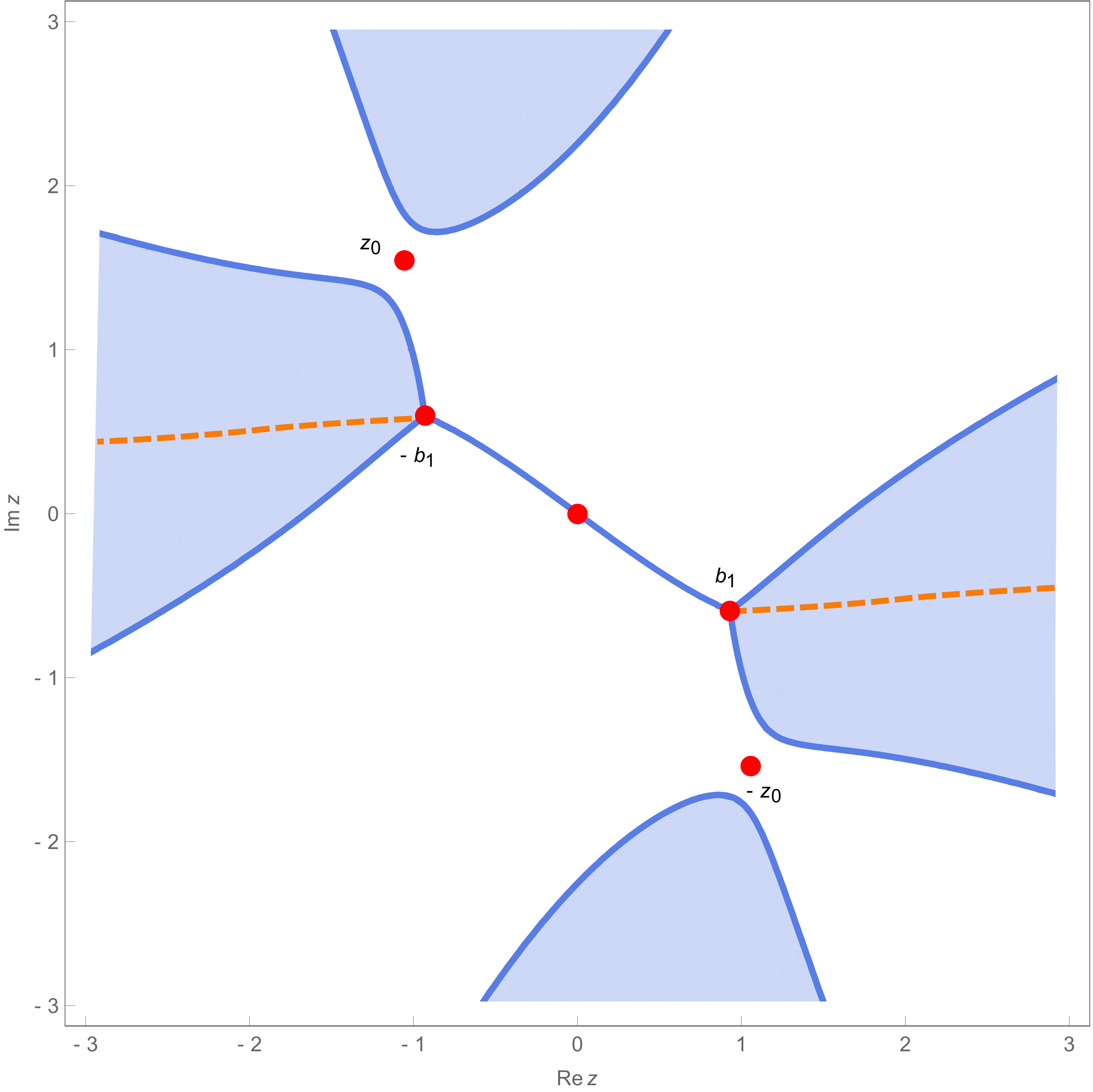}
		\caption{$\sigma=1+3.8\ii \in \mathcal{O}_1$.}
		\label{fig:stable barren 1 3.8}    
	\end{subfigure}
	\begin{subfigure}{.24\textwidth}
		\centering
		\includegraphics[width=\textwidth]{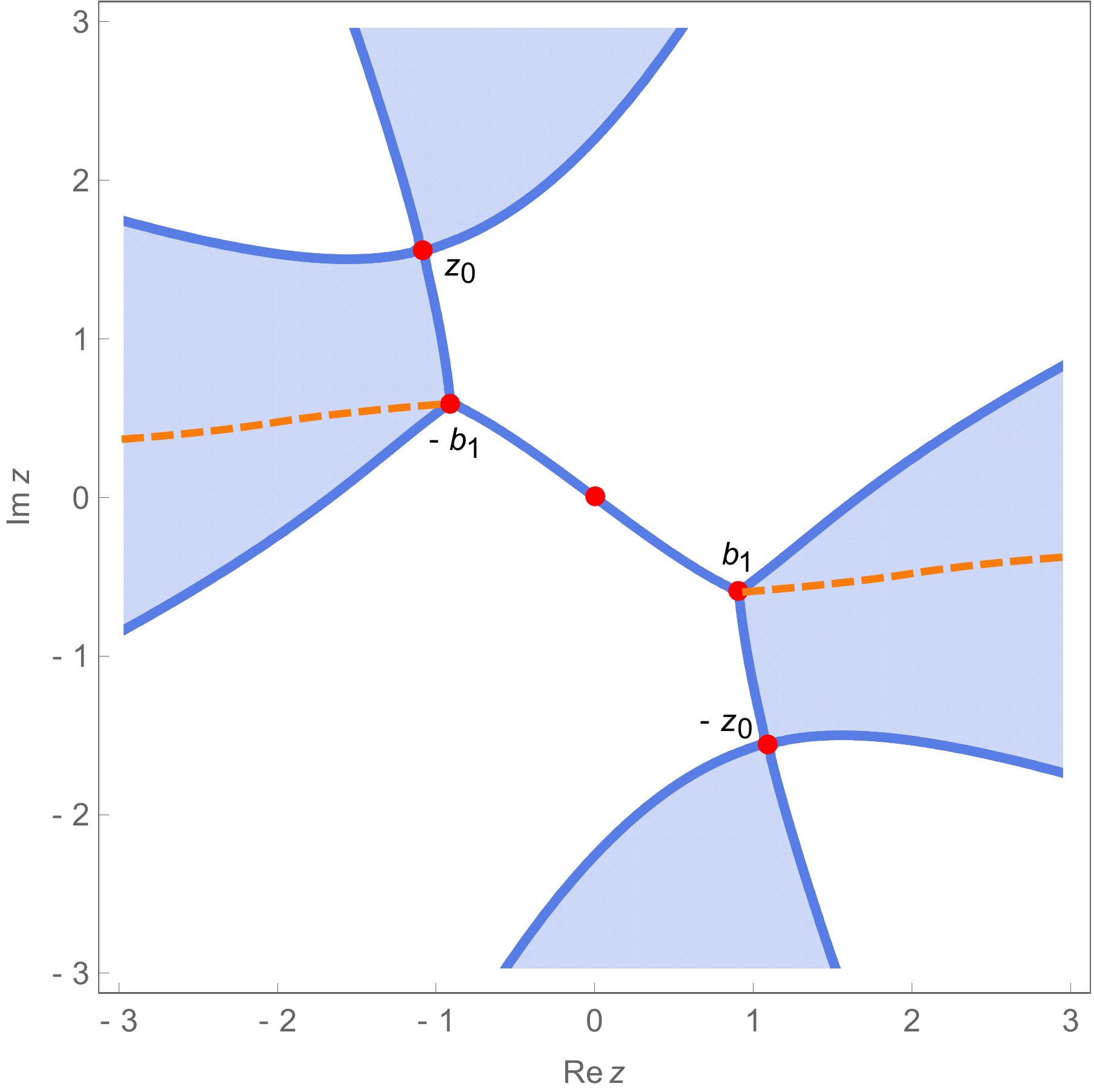}
		\caption{$\sigma=1+3.92\ii \in \mathcal{O}_1$.}
		\label{fig:stable barren critical}    
	\end{subfigure}
	\begin{subfigure}{.24\textwidth}
		\centering
		\includegraphics[width=\textwidth]{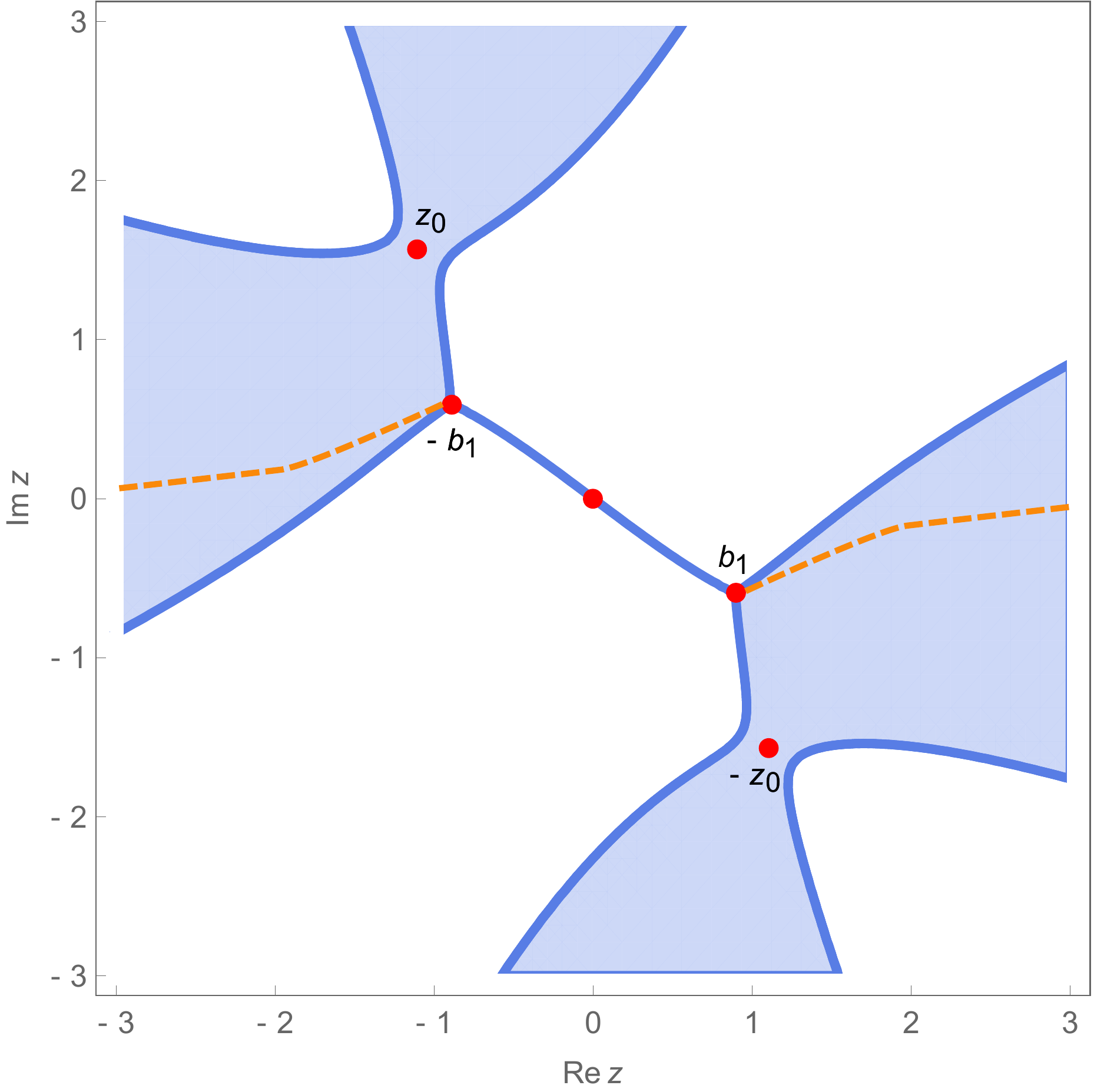}
		\caption{$\sigma=1+4\ii \in \mathcal{O}_1$.}
		\label{fig:stable barren 1 4}    
	\end{subfigure}
	
	\begin{subfigure}{0.24\textwidth}
		\centering
		\includegraphics[width=\textwidth]{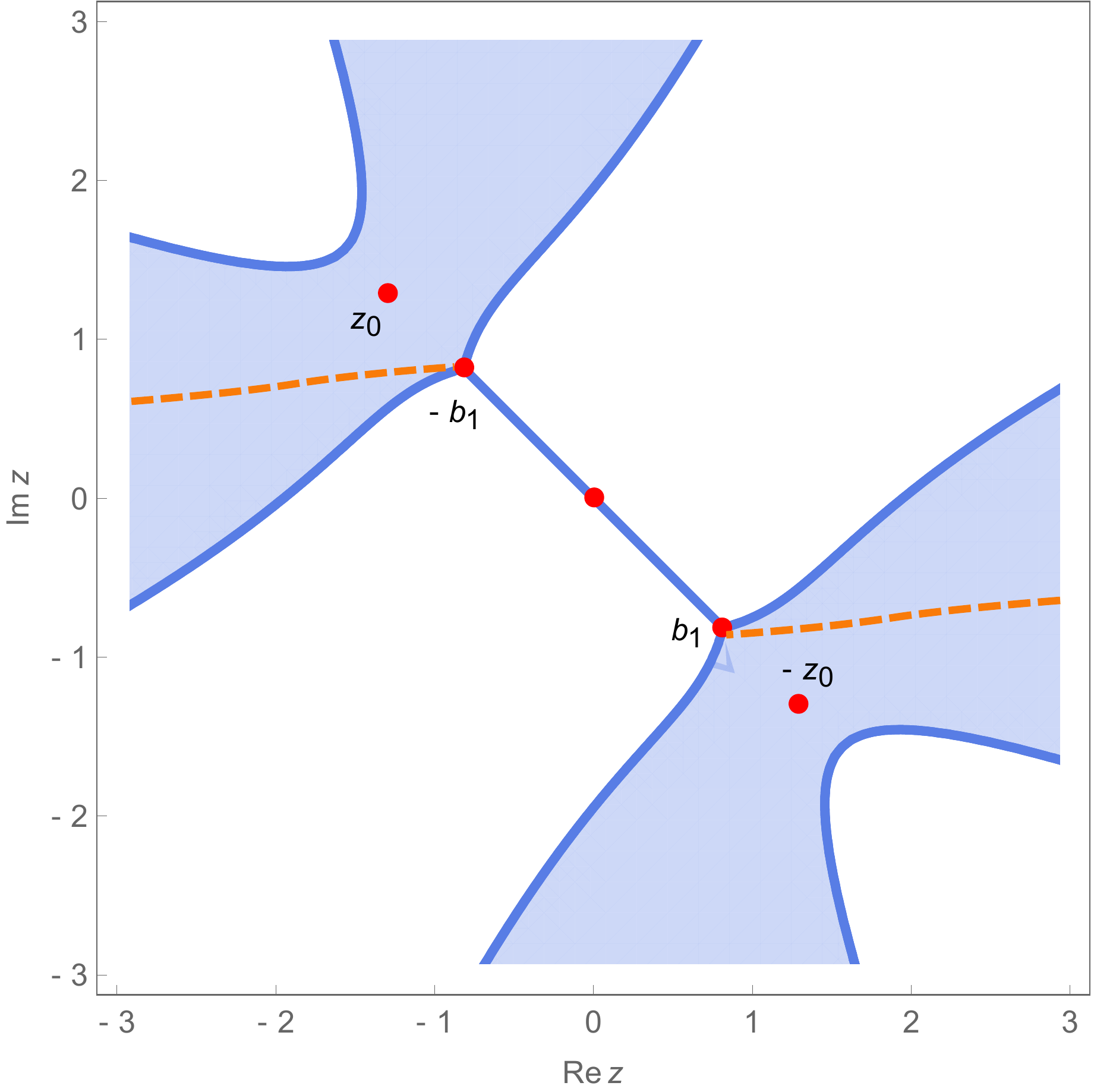}
		\caption{$\sigma=4 \ii \in \mathcal{O}_1$.}
		\label{fig:fertile barren 4}
	\end{subfigure}%
	\begin{subfigure}{.24\textwidth}
		\centering
		\includegraphics[width=\textwidth]{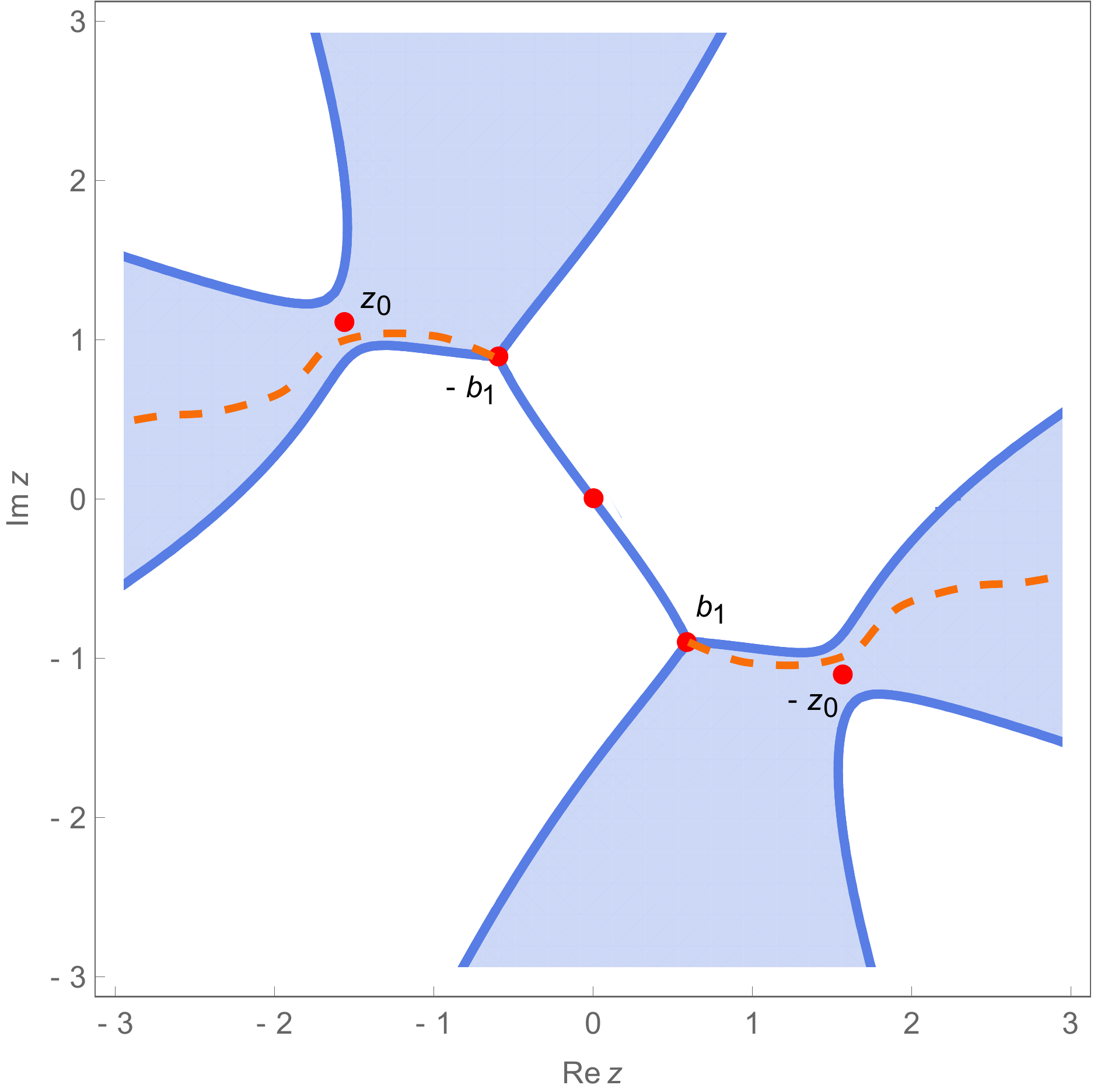}
		\caption{$\sigma=-1+4\ii \in \mathcal{O}_1$.}
		\label{fig:stable barren -1 4}    
	\end{subfigure}
	\begin{subfigure}{.24\textwidth}
		\centering
		\includegraphics[width=\textwidth]{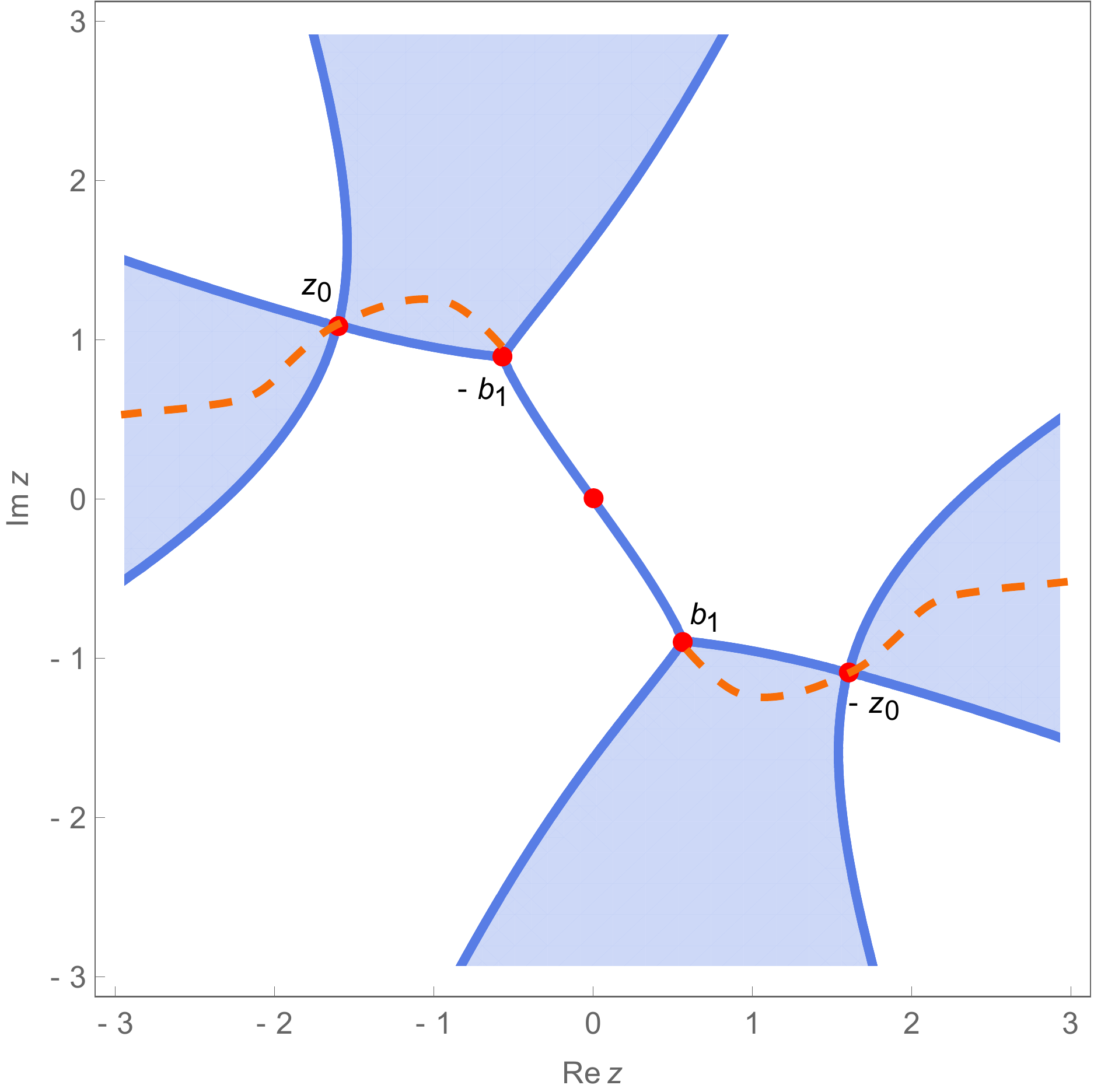}
		\caption{$\sigma_{\mbox{cr}}\simeq -1.15+4\ii \notin \mathcal{O}_1$.}
		\label{fig:stable barren -1.15 4}    
	\end{subfigure}
	\begin{subfigure}{.24\textwidth}
		\centering
		\includegraphics[width=\textwidth]{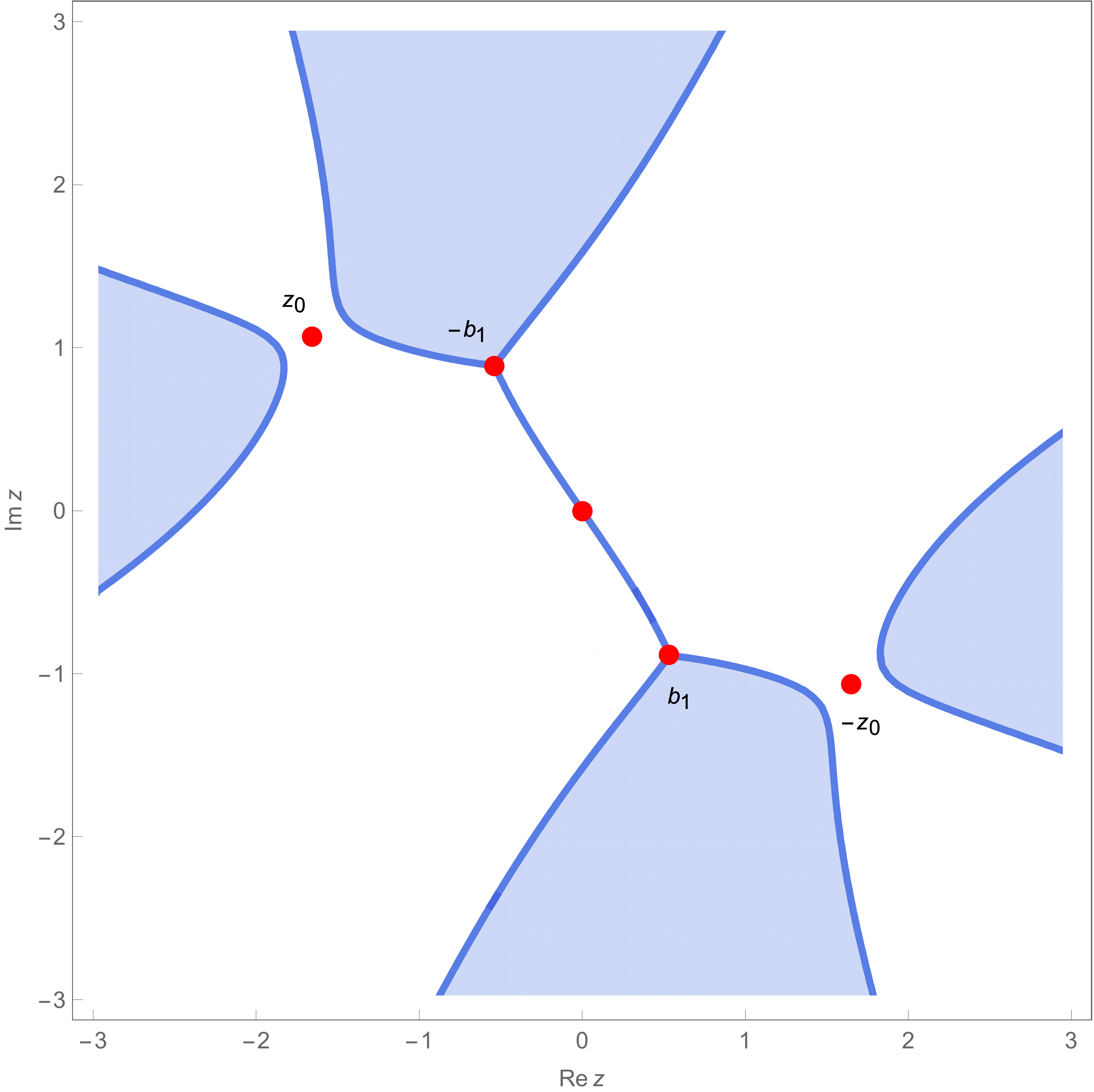}
		\caption{$\sigma=-1.35+4\ii\notin \mathcal{O}_1$.}
		\label{fig:stable barren -1.35 41}    
	\end{subfigure}
	\caption{
		This sequence of figures shows allowable regions in light blue through which the contour of integration (for the orthogonal polynomials) must pass, for a varying collection of values of $\sigma$.	Regions in light blue are the $\sigma$-stable lands where $  \Re[\eta_1(z;\sigma)]<0$ and the regions in white are the $\sigma$-unstable lands where $  \Re[\eta_1(z;\sigma)]>0$ (see \eqref{Re eta < 0} and \eqref{eta q}).}
	\label{fig:stable and barren}
\end{figure}

It should also be mentioned that  transitions through these boundaries correspond qualitatively to phase transitions in the asymptotic behavior of the orthogonal polynomials.  For example, in the simpler case of real potentials, if there is one contour comprising the support, the oscillatory behavior of the polynomails is expressed via trigonometric functions  \cite{Deiftetal}, while if there are several intervals, then the oscillatory behavior is descirbed by a Jacobi theta function associated to the Riemann surface of $R(z;\boldsymbol{t})$ \cite{Deiftetal2}. 

\begin{comment}
	
So we refer to the \textit{regular $q$-cut regime} as the set of all $\boldsymbol{t} \in \C^{2p-1}$ such that:
\begin{itemize}
	\item[a)] The support $J_{\boldsymbol{t}}$ of the equilibrium measure is comprised of $q$ disjoint analytic arcs whose endpoints are all simple zeros of the polynomial $Q(z;\boldsymbol{t})$, 
	\item[b)]  the other zeros of $Q(z;\boldsymbol{t})$, if any, do not lie on the support of the equilibrium measure, and 
	\item[c)]  there exists a  complementary set $\ga_{\boldsymbol{t}}$ which lies \textit{entirely} in the set \begin{equation*}
	\left\{ z : \Re \left[- \int^z_{b_q(\boldsymbol{t})} \sqrt{Q(s;\boldsymbol{t})} \dd s \right]  <0 \right\}.
	\end{equation*} such that $\Ga_0 \equiv \overline{\ga_0 \cup J}$.
\end{itemize}
\end{comment}

The main purpose of this work is to present a brief self-contained proof of the fact that if for some $\boldsymbol{t}^* \in \C^{2p-1}$ the corresponding equilibrium measure is $q$-cut regular, then there exists a small enough neighborhood $D_\ep(\boldsymbol{t}^*)$ of $\boldsymbol{t}^*$  so that for all $\boldsymbol{t} \in D_\ep(\boldsymbol{t}^*) $ the associated equilibrium measures are also $q$-cut regular. Lemma 4.2 of \cite{BertolaTovbis2016} gives another proof of  the openness of regular set of parameters using the determinantal form of the function $\eta_q$, and uses arguments from \cite{TVZ,TV}.  The proof that we present here avoids computations of the Jacobian determinant, but rather has the flavor of a vanishing lemma from the theory of Riemann-Hilbert problems, which permits us to arrive at a contradiction if the Jacobian determinant should vanish at a regular point. Some of these arguments use ideas based on Riemann surface theory contained in \cite{Ber} in which Lemma 4.1 provides a proof. Indeed, in \S \ref{Sec Openness} we prove: 
\begin{theorem}\label{main thm}
	The regular  $q$-cut regime is open.
\end{theorem}
The proof of Theorem \ref{main thm} relies upon showing that the underlying equations for finding the endpoints are solvable for every $\boldsymbol{t} \in D_\ep(\boldsymbol{t}^*)$. To this end in \S \ref{SecSolvability} we formulate the end-point equations in the $q$-cut case and prove the following result:
\begin{theorem}\label{solvability thm}
	The equations which determine the $2q$ endpoints of the regular $q$-cut regime are solvable for all $\boldsymbol{t}$ in a small enough neighborhood of a regular q-cut point $\boldsymbol{t}^*$, all endpoints $a_j(\boldsymbol{t}), b_j(\boldsymbol{t})$ are distinct, and $\Re a_j(\boldsymbol{t})$, $\Im a_j(\boldsymbol{t})$, $1\leq j \leq q$ are real-analytic functions of $\Re t_k$, $\Im t_k$, $1\leq k\leq 2p-1$.
\end{theorem}

Another important ingredient in the proof of Theorem \ref{main thm}, mainly useful for establishing that the \textit{regularity} properties are preserved for every $\boldsymbol{t} \in D_\ep(\boldsymbol{t}^*)$, is the continuity of the critical graph of the associated quadratic differential which, in particular, has within itself the $q$-cut support of the equilibrium measure. Apart from the continuity of the support $J_{\boldsymbol{t}}$, knowing the continuity of the complementary part of the critical graph, i.e. $\mathscr{J}_{\boldsymbol{t}} \setminus J_{\boldsymbol{t}}$, is also very important. This is because the \textit{"closure of a strait"} (recall, for example, the passage from $\sigma=-1+4\ii$ to $\sigma_{\mbox{cr}} \simeq -1.15+4\ii$ depicted in Figures \ref{fig:stable barren -1 4} and \ref{fig:stable barren -1.15 4}) is directly tied to the behavior of the complementary part $\mathscr{J}_{\boldsymbol{t}} \setminus J_{\boldsymbol{t}}$ of the critical graph, which leads to the impossibility of having complementary contours $\Ga_{\boldsymbol{t}} \setminus J_{\boldsymbol{t}}$ to lie entirely in the $\boldsymbol{t}$-stable lands (see the orange dashed line in Figure \ref{fig:stable barren -1.15 4}). To that end, in \S \ref{Sec Openness}, for the entirety of the critical graph $\mathscr{J}_{\boldsymbol{t}}$ we prove:

\begin{theorem}\label{continuous deformations of one cut critical graph1}
	The critical graph $\mathscr{J}_{\boldsymbol{t}}$ of the quadratic differential $$Q(z;\boldsymbol{t})\dd z^2 \equiv \left(-\om(z;\boldsymbol{t})+\frac{V'(z;\boldsymbol{t})}{2}\right)^2\dd z^2, $$ and thus the support $J_{\boldsymbol{t}}$ of the equilibrium measure, deform continuously with respect to $\boldsymbol{t}$.
\end{theorem}

\section{Equilibrium Measure and Quadratic Differentials}

\subsection{Equilibrium measure for orthogonal polynomials associated with real external fields.}\label{EqMeasRealCase} Let 
\[
V(x;\boldsymbol{t})=\frac{x^{2p}}{2p}+\sum_{j=1}^{2p-1} \frac{t_j z^j}{j},
\]
be any polynomial of even degree with real coefficients. Now consider the following energy functional which is defined on the space of probability measures on $\R$:
\begin{equation}\label{em1a}
I_{V}(\nu):= \underset{\R}{\int}\underset{\R}{\int} \log \frac{1}{|x-y|}\dd \nu(x) \dd \nu(y) + \int_{\R} V(x) \dd \nu(x),
\end{equation}
The equilibrium measure, $\nu_{\rm eq}$, is a {\it probability measure} on $\R$ which achieves the infimum of the above functional:
\begin{equation}\label{em2a}
\underset{\mathscr{M}_1(\R)}{\inf} I_{V}(\nu) = I_{V}(\nu_{\rm eq}),
\end{equation}
where 
\[ 
\mathscr{M}_1(\R):= \left\{ \nu: \nu \geq 0, \ \int_{\R} \dd \nu = 1  \right\}. 
\]
For this extremal problem, it is known that (see, e.g., \cite{DKM, Deiftetal2, BL})
\bge
\item The equilibrium measure {\it exists} and is {\it unique}.
\item The equilibrium measure is {\it absolutely continuous} with respect to the Lebesgue measure,
\[
\dd \nu_{\rm eq}(z)=\rho_V(z)\,\dd z.
\]
\item
The  support of $\nu_{\rm eq}$ consists of {\it finitely many closed intervals},
\[
J=\supp \nu_{\rm eq}=\bigcup_{k=1}^{q} [a_k,b_k],
\]
where $q\le p$.
The intervals $\{[a_k,b_k],\; k=1,\ldots,q\}$ of the support of $\nu_{\rm eq}$ are called the {\it cuts}. We may assume that $a_1<b_1<a_2<b_2<\ldots<a_q<b_q$.
\item The density of the equilibrium measure on the support $J$ can be written in the form,
\begin{equation}\label{em3a}
\rho_V(x)=\frac{1}{2\pi i}\,h(x)R_+^{1/2}(x),\quad R(x)=\prod_{k=1}^q(x-a_k)(x-b_k),
\end{equation}
where $h(x)$ is a polynomial, such that $h(x)\ge 0$ for all $x\in J$, 
and $R^{1/2}(x)$ is the branch on the complex plane of the square root of $R(x)$, with cuts on $J$, which is positive for large positive $x$. Respectively,
$R_+^{1/2}(x)$ is the value of $R^{1/2}(x)$ on the upper part of the cut.
\item Finally, the polynomial $h(x)$ is the {\it polynomial part} of the function $\displaystyle \frac{V'(x)}{R^{1/2}(x)}$ at infinity, i.e.,
\begin{equation}\label{em4a}
\frac{V'(x)} {R^{1/2}(x)} =h(x)+\mathcal O(x^{-1}).
\end{equation}
This determines $h(x)$ and hence the equilibrium measure $\nu_{\rm eq}$ uniquely, as long  as we know the end-points $a_1,b_1,\ldots,a_q,b_q$.
\ene

An important property of this minimization problem \eqref{em2a} is
that the minimizer $\nu_{\rm eq}$ is {\it uniquely determined} by the Euler--Lagrange variational conditions: 
\begin{equation}\label{em5a}
2\int_{\R}\log |x-y| \,\dd \nu_V(y) -V(x)=l,\quad \textrm{for}
\quad x\in J,
\end{equation}
\begin{equation}\label{em6a}
2\int_{\R}\log |x-y| \,\dd \nu_V(y) -V(x)\le l,\quad \textrm{for}
\quad x\in \R\setminus J,
\end{equation}
for some real constant {\it Lagrange multiplier} $l$, which is the 
{\it same for all cuts} $[a_k,b_k]$. From this we conclude that 
\begin{equation}\label{em7a}
\int_{b_k}^{a_{k+1}} h(x)R^{1/2}(x)\,\dd x=0, \quad
k=1,\ldots, q-1.
\end{equation}
Therefore the polynomial $h(x)$ has a zero on every interval $[b_k,a_{k+1}]$,
which means that $\deg h\ge q-1$.

We also consider the {\it resolvent} of the equilibrium measure defined as 
\begin{equation}\label{em8a}
\om(z)=\int_J\frac{\dd \nu_V(x)}{z-x}\,,\quad z\in \C\setminus J.
\end{equation}
This function, which is very useful to construct the density of the equilibrium measure, has the following analytical and asymptotic properties:
\bge
\item $\om(z)$ is analytic on the set $\C\setminus J$.
\item On $J$, the equilibrium condition \eqref{em5} implies that
\begin{equation}\label{em9a}
\om_+(x)+\om_-(x)=V'(x),
\end{equation}
and the Plemelj--Sokhotski formula implies
\begin{equation}\label{em10a}
\om_+(x)-\om_-(x)=-2\pi i\rho_V(x).
\end{equation}
Combining these equation with formula \eqref{em3} for $\rho_V(x)$, we obtain that
\begin{equation}\label{em11a}
\om(z)=\frac{V'(z)}{2}-\frac{h(z)R^{1/2}(z)}{2}.
\end{equation}
\item As $z\to\infty$,
\begin{equation}\label{em12a}
\om(z)=\frac{1}{z}+\frac{m_1}{z^2}+\ldots,\quad 
m_k=\int_Jx^k\rho_V(x)\,dx.
\end{equation}
\ene

\subsection{Equilibrium measure for orthogonal polynomials associated with complex external fields.}\label{EqMeasCplxCase}
In this section we follow the work of Kuijlaars and Silva \cite{KS} (See also \cite{Ber,MFRakh1,Rakh}). Let us consider the general complex external field of even degree 
\begin{equation}\label{em1}
V(z)=\frac{z^{2p}}{2p}+\sum_{j=1}^{2p-1} \frac{t_j z^j}{j}, \qquad t_j \in \C, \qquad j=1, \cdots, 2p-1.
\end{equation}

For $0<\ep< \pi/4p$, consider the sectors
\begin{equation}\label{em2}
\bga
S^+_{\ep}&=\left\{z\in\C\;\Big |\; |\arg z|\le \frac{\pi}{4p}-\ep\right\}\\
S^-_{\ep}&=\left\{z\in\C\;\Big |\; |\arg z-\pi|\le \frac{\pi}{4p}-\ep\right\}.
\ena
\end{equation}
Observe that in these sectors we particularly have,
\begin{equation}\label{em3}
\lim_{z\to\infty} \Re V(z)=\infty.
\end{equation} By a contour we mean a {\it continuous curve} $z=z(t)$, $-\infty<t<\infty$, without self-intersections, and we say that a contour $\Ga$ is \textit{admissible} if
\bge
\item The contour $\Ga$  is a finite union of  $C^1$ Jordan arcs.
\item There exists $\ep>0$ and $r_0>0$, such that $\Ga$ goes from $S^-_\ep$ to $S^+_\ep$ in the sense that $\forall\,r>r_0,$  
$\exists\, t_0<t_1$ such that 
\[
z(t)\in S^-_\ep \setminus D_r \quad \forall\, t<t_0;\quad 
z(t)\in S^+_\ep \setminus D_r \quad \forall\, t>t_1,
\]
where $D_r$ is the disk centered at the origin with radius $r$. We will assume that the contour $\Ga$ is oriented from $(-\infty)$ 
to $(+\infty)$, where $(-\infty)$ lies in the sector $S^-_\ep$ and
$(+\infty)$ in the sector $S^+_\ep$. The orientation defines an order
on the contour $\Ga$.
\ene
An example of an admissible contour is the real line. We denote the collection of all admissible contours by $\mathcal{T}$.

For $\Ga\in\mathcal T$, let $\mathcal P(\Ga)$ be the space of probability measures $\nu$ on $\Ga$, satisfying
\begin{equation}\label{em4}
\bga
\underset{\Ga}{\int} &  |\Re V(s)|\, \dd \nu(s)<\infty.
\ena
\end{equation}
Consider the following real-valued energy functional on $\mathcal P(\Ga)$:
\begin{equation}\label{em5}
I_{V,\Ga}(\nu):= \underset{\Ga\,\times\,\Ga}{\iint} \log \frac{1}{|z-s|}\,\dd \nu(z) \dd \nu(s) + \int_{\Ga} \Re V(s)\, \dd \nu(s).
\end{equation}
Then there exists a unique minimizer $\nu_{V,\Ga}$ of this functional (see \cite{SaffTotik}) so that
\begin{equation}\label{em6}
\min_{\nu\in \mathcal P(\Ga)} I_{V,\Ga}(\nu)= I_{V,\Ga}(\nu_{V,\Ga}).
\end{equation}

The minimizing probability measure $\nu_{V,\Ga}$ is referred to as the
{\it equilibrium measure} of the functional $I_{V,\Ga}(\nu)$, and
its support is a compact set $J_{V,\Ga}\subset \Ga$, and is uniquely determined by the {\it Euler--Lagrange variational conditions}. Namely,
$\nu_{V,\Ga}$ is the unique probability measure $\nu$ on $\Ga$ such that
there exists a constant $l$, the Lagrange multiplier, such that
\begin{equation}\label{em7}
\bga
& U^{\nu}(z)+\frac{1}{2}\,\Re V(z)=\ell,\quad z\in \supp \nu,\\
& U^{\nu}(z)+\frac{1}{2}\,\Re V(z)\ge \ell,\quad z\in \Ga\setminus \supp \nu,
\ena
\end{equation}
where 
\begin{equation}\label{em8}
U^{\nu}(z)=\int_{\Ga} \log \frac {1}{|z-s|}\,\dd \nu(s) 
\end{equation}
is the {\it logarithmic potential} of the measure $\nu$ \cite{SaffTotik}. 

Now we maximize the minimized energy functional $I_{V}(\nu_{V,\Ga})$ over all admissible contours $\Ga\in \mathcal T$. In \cite{KS}, the authors prove that the maximizing contour $\Ga_0\in\mathcal T$ exists, and
the equilibrium measure 
\[
\nu_{\rm eq} \equiv \nu_{V,\Ga_0}
\]
 is supported by a set $J\subset\Ga_0$ which is a finite union 
of {\it analytic arcs} $\Ga_0[a_k,b_k]\subset \Ga_0,$\footnote{Given two points $s_1,s_2\in \C \cup\{\pm \infty\}$ on $\Ga_{\boldsymbol{t}}$, by $\Ga_{\boldsymbol{t}}(s_1,s_2)$ and $\Ga_{\boldsymbol{t}}[s_1,s_2]$ we respectively denote the open and closed "intervals" on $\Ga_{\boldsymbol{t}}$ starting at $s_1$ and ending at $s_2$.}$\;k=1,\ldots,q$, 
\[
J=\bigcup_{k=1}^q \Ga_0[a_k,b_k],\quad a_1<b_1\le a_2<b_2\le \ldots\le a_q<b_q,
\] 
that are {\it critical trajectories}
of a quadratic differential\footnote{See  \S \ref{SecQD} for a review of definitions and basic facts about quadratic differentials.} $Q(z)\,\dd z^2$, where $Q(z)$ is a polynomial of degree 
\begin{equation}\label{em9}
\deg Q(z)=2\deg V(z)-2=4p-2.  
\end{equation}
\begin{comment}
The fact that the arcs $\Ga_0[a_k,b_k],\;k=1,\ldots,q,$
are critical trajectories of the quadratic differential $Q(z)\,\dd z^2$ means that 
\bge
\item 
\[
Q(a_1)=Q(b_1)=\ldots=Q(a_q)=Q(b_q)=0, 
\] 
and
\item
\[
-Q(z)\,\dd z^2>0,\quad \forall\, z\in \bigcup_{k=1}^q (a_k,b_k).
\]
\ene
Observe that the second property implies that $Q(z)\not=0$ on the open arcs $(a_k,b_k)$, $k=1,\ldots,q$.
\end{comment}

Moreover, in \cite{KS} it is proven that the polynomial $Q(z)$ is equal to
\begin{equation}\label{em10}
Q(z)=\left(-\om(z)+\frac{V'(z)}{2}\right)^2, 
\end{equation}
where 
\begin{equation}\label{em11}
\om(z)= \int_{J} \frac{ \dd \nu_{\rm eq}(s)}{z-s}
\end{equation}
is the resolvent of the measure $\nu_{\rm eq}$. From
\[
\frac{1}{z-s}=\frac{1}{z}+\frac{s}{z^2}+\frac{s^2}{z^3}+\ldots,
\]
we obtain that $\om(z)=\mathcal{O}(z^{-1})$ as $z\to\infty$:
\begin{equation}\label{em12}
\om(z)=\frac{1}{z}+\frac{m_1}{z^2}+\ldots,\quad \mbox{with} \qquad m_k=\int_J s^k
\dd \nu_{\rm eq}(s).
\end{equation}

Additionally, the equilibrium measure $\nu_{\rm eq}$ is absolutely continuous 
with respect to the arc length. More precisely we have
\begin{equation}\label{em13}
\dd \nu_{\rm eq}(s)=\frac{1}{\pi \ii}\,Q_+(s)^{1/2} \dd s,
\end{equation}
where $Q_+(s)^{1/2}$ is the limiting value of the function
\begin{equation}\label{em14}
Q(z)^{1/2}=-\int_{J} \frac{ \dd \nu_{\rm eq}(s)}{z-s}+\frac{V'(z)}{2}, 
\end{equation}
as $z\to s\in J$ from the left-hand side of $J$ with respect to the orientation of the contour $\Ga_0$ from $(-\infty)$ to $\infty$. A very important resultin \cite{KS} is that
the equilibrium measure $\nu_{\rm eq}$ is {\it unique} as the \textit{max-min measure}. On the other hand, the infinite contour $\Ga_0$ is not unique because it can be deformed outside of the support $J$ of $\nu_{\rm eq}$, as long as $\Ga_0 \setminus J$ lies in the $\boldsymbol{t}$-stable lands.

\subsection{The $g$-function.}
 As usual we define the "$g$-function" as
\begin{equation}\label{em16}
g(z)= \int_{J} \log(z-s)\,\dd \nu_{\rm eq}(s),
\end{equation}
where for a fixed $s\in J$, we consider a cut of $\log(z-s)$ to be
$\Ga_0(-\infty,s]$. Notice that by \eqref{em11} we have
\begin{equation}\label{em17}
g'(z)= \int_{J}\frac{\dd \nu_{\rm eq}(s)}{z-s}=\om(z).
\end{equation}
Moreover, from \eqref{em8}, the logarithmic potential $U^{\nu_{\rm eq}}(z)$ can be written as
\begin{equation}\label{em18}
  U^{\nu_{\rm eq}}(z)=\int_{J} \log \frac {1}{|z-s|}\,\dd \nu_{\rm eq}(s) 
  =-\Re g(z),
\end{equation}
and therefore the Euler-Lagrange variational conditions \eqref{em7} can be expressed as
\begin{equation}\label{em19}
\bga
& -\Re g(z)+\frac{1}{2}\,\Re V(z)=\ell,\quad z\in J,\\
& -\Re g(z)+\frac{1}{2}\,\Re V(z)\ge \ell,\quad z\in \Ga_0\setminus J.
\ena
\end{equation}

\subsection{Regular and Singular Equilibrium Measures} An equilibrium measure $\nu_{\rm eq}$ is called {\it regular} if the following three conditions hold:
\bge
\item The arcs $\Ga_0[a_k,b_k],\;k=1,\ldots,q,$ of the support of $\nu_{\rm eq}$ are disjoint.
\item The end-points $\{a_k,b_k,\;k=1,\ldots,q\}$ are simple zeros of the polynomial $Q(s)$.
\item There is a contour $\Ga_0$ containing the support $J$ of $\nu_{\rm eq}$
such that 
\begin{equation}\label{em20}
U^{\nu}(z)+\frac{1}{2}\,\Re V(z)> \ell,\quad z\in \Ga_0\setminus J.
\end{equation}
\ene
An equilibrium measure $\nu_{\rm eq}$ is called {\it singular} (or
{\it critical}) if it is not regular.

\subsubsection{Regular Equilibrium Measures}  Assume that
the equilibrium measure $\nu_{\rm eq}$ is regular.
Because the resolvent 
\begin{equation}\label{em21}
\om(z)= \int_{J} \frac{ \dd \nu_{\rm eq}(s)}{z-s}
\end{equation}
is analytic on $\C\setminus J$,
one can see from equation \eqref{em10} that all the zeros of the
polynomial $Q(z)$ different from the end-points $\{a_k,b_k,\;k=1,\ldots,q\}$ must be of even degree, and thus
$Q(z)$ can be expressed as
\begin{equation}\label{em22}
Q(z)=\frac{1}{4}\,h(z)^2R(z),
\end{equation}
where $h(z)$ is some polynomial,
\begin{equation}\label{em23}
h(z)=\prod_{j=1}^r (z-z_j),
\end{equation}
having zeros $z_1,\ldots,z_r$ which are distinct from the $2q$ end-points $\{a_k,b_k\}^{q}_{k=1}$,
and
\begin{equation}\label{em24}
R(z)=\prod_{k=1}^q (z-a_k)(z-b_k).
\end{equation}
Therefore,
\begin{equation}\label{em25}
Q(z)=\frac{1}{4}\,h(z)^2R(z)=\frac{1}{4}\prod_{j=1}^r (z-z_j)^2
\prod_{k=1}^q (z-a_k)(z-b_k).
\end{equation}
In \eqref{em23}, and \eqref{em25} if $r=0$, it is understood that $h(z)\equiv 1$. By taking the square root with the plus sign, we obtain that
\begin{equation}\label{em26}
Q(z)^{1/2}=\frac{1}{2}\,h(z)R(z)^{1/2}
=\frac{1}{2}\prod_{j=1}^r (z-z_j)\left[\prod_{k=1}^q (z-a_k)(z-b_k)\right]^{1/2},
\end{equation}
Correspondingly, equation \eqref{em13} can be rewritten as
\begin{equation}\label{em27}
\bga
\dd \nu_{\rm eq}(z)&=\frac{1}{2\pi \ii}\,h(z)R_{+}(z)^{1/2} \dd z =\frac{1}{2\pi \ii}\,\prod_{j=1}^r (z-z_j)\left[\prod_{k=1}^q (z-a_k)(z-b_k)\right]_+^{1/2} \dd z.
\ena
\end{equation}
From \eqref{em10}, \eqref{em17}, and  \eqref{em26} we can write

\begin{equation}
g(z;\boldsymbol{t}) = \frac{V(z;\boldsymbol{t})+\ell^{(q)}_*(\boldsymbol{t})}{2}+\frac{\eta_q(z;\boldsymbol{t})}{2} , \qquad z \in \C \setminus \Ga_{\boldsymbol{t}}(-\infty,b_q(\boldsymbol{t})],
\end{equation}
where 
\begin{equation}\label{eta def}
\eta_q(z;\boldsymbol{t}):= -\int^{z}_{b_q(\boldsymbol{t})} \prod_{\ell=1}^r \left(s-z_\ell(\boldsymbol{t})\right)\left[\prod_{j=1}^q \left(s-a_j(\boldsymbol{t})\right)\left(s-b_j(\boldsymbol{t})\right)\right]^{1/2} \dd s, \qquad z \in \C \setminus \Ga_{\boldsymbol{t}}(-\infty,b_q(\boldsymbol{t})],
\end{equation}
in which the path of integration does not cross $\Ga_{\boldsymbol{t}}(-\infty,b_q(\boldsymbol{t})]$. Also from \eqref{em10} and \eqref{em26} we have \begin{equation}\label{g'++g'-}
g'_{+}(z;\boldsymbol{t})+ g'_{-}(z;\boldsymbol{t}) = V'(z;\boldsymbol{t}), \qquad  \ z \in J_{\boldsymbol{t}} = \bigcup_{j=1}^{q} \Ga_{\boldsymbol{t}} (a_{j}(\boldsymbol{t}),b_{j}(\boldsymbol{t})).
\end{equation}
We use \eqref{em17} and \eqref{em22} to rewrite \eqref{em10} as:
\begin{equation}\label{g' V' h sqrt R}
g'(z;\boldsymbol{t}) = \frac{1}{2} \left[ V'(z;\boldsymbol{t}) - h(z;\boldsymbol{t}) R^{1/2}(z;\boldsymbol{t})\right].
\end{equation}
\subsection{Quadratic Differentials}\label{SecQD}
In this subsection we briefly remind some definitions and basic facts about quadratic differentials from \cite{Strebel}. The zeros and poles of $Q(z)$ are referred to as the \textit{critical points} of the quadratic differential $Q(z)\dd z^2$, and all other points are called \textit{regular points} of $Q(z)\dd z^2$. For some fixed value $ \theta \in [0, 2\pi)$, the smooth curve $L_{\theta}$ along which \begin{equation}
\arg Q(z)\dd z^2 = \theta, 
\end{equation}
is defined as the $\theta$-arc of the quadratic differential $Q(z)\dd z^2$, and a maximal $\theta$-arc is called a $\theta$-trajectory.
The above equation implies that a $\theta$-arc can only contain regular points of $Q$, because at the critical points $\arg Q(z)$ is not defined. For a meromorphic quadratic differential, there is only one $\theta$-arc passing through each regular point. 

 We will refer to a $\pi$-trajectory ( resp. $0$-trajectory) which is incident with a critical point as a \textit{critical trajectory} (resp. \textit{critical orthogonal trajectory}). If $b$ is a critical point of $Q(z)\dd z^2$, then the totality of the solutions to 
\begin{equation} 
\Re\left(\int^z_b \sqrt{Q(s)} \dd s \right) = 0,
\end{equation}
is referred to as the \textit{critical graph} of $\int^z_b \sqrt{Q(s)} \dd s$ which is referred to as the \textit{natural parameter} of the quadratic differential $Q(z)\dd z^2$  (see \S 5 of \cite{Strebel}). A Jordan curve $\Sigma$ composed of open $\theta$-arcs and their endpoints, with respect to some meromorphic quadratic differential $Q(z)\dd z^2$, is a simple closed  \textit{geodesic polygon}  (also referred to as a $Q$-polygon). The endpoints may be regular or critical points of $Q(z)\dd z^2$, which form the vertices of the $Q$-polygon. $\Sigma$ is called a \textit{singular geodesic polygon}, if at least one of its end points is a singular point.

Now we can state the  \textit{Teichm\"uller's lemma}: for a meromorphic quadratic differential $Q(z)\dd z^2$, assume that $\Sigma$ is a $Q$-polygon, and let $\texttt{V}_{\Sigma}$ and $\texttt{Int} \Sigma$  respectively denote its set of vertices and interior. Then
\begin{equation}\label{Teich Lemma}
\#\texttt{V}_{\Sigma} - 2 = \sum_{z \in \texttt{V}_{\Sigma}} (\texttt{ord}(z)+2)\frac{\theta(z)}{2\pi} + \sum_{z \in \texttt{Int} \Sigma} \texttt{ord}(z),
\end{equation}
where $\theta(z)$ denotes the interior angle of $\Sigma$ at $z$, and $\texttt{ord}(z)$ is the order of the point $z$ with respect to the quadratic differential. That is, $\texttt{ord}(z)=0$ for a regular point, $\texttt{ord}(z)=n$  if $z$ is a zero of order $n \in \N$, and $\texttt{ord}(z)=-n$  if $z$ is a pole of order $n \in \N$ of the quadratic differential. We use the Teichm\"uller's lemma in the proof of Theorem \ref{main thm} in \S \ref{Sec Openness}.

\section{Endpoint Equations and the  Regular $q$-cut Regime}

Notice that from \eqref{g'++g'-} we have \begin{equation*}
     R_+^{-1/2}(z)g'_+(z)  = - R^{-1/2}_-(z) \left(V'(z)-g'_-(z)\right) =  R_-^{-1/2}(z)g'_-(z) -  R_-^{-1/2}(z)V'(z) 
\end{equation*}
Therefore by Plemelj-Sokhotskii we have 
\begin{equation}\label{g' series}
    g'(z) = \frac{R^{1/2}(z)}{2 \pi \ii } \int_{J} \frac{V'(s)}{R^{1/2}_{+}(s)} \frac{\dd s}{s-z} = - \frac{R^{1/2}(z)}{2 \pi \ii z} \sum_{\ell=0}^{\infty} \frac{T_{\ell}}{z^{\ell}}
\end{equation}
where
\begin{equation}\label{moments}
    T_{\ell} = \int_{J} \frac{V'(s)}{R^{1/2}_{+}(s)} s^{\ell} \dd s \ , \ \qquad \ \ell \in \N \cup \{0\}.
\end{equation}

From \eqref{g' series} and the requirement that $g'(z)=z^{-1}+O(z^{-2})$ as $z \to \infty$, we obtain the following $q+1$ equations:

\begin{equation}\label{moment conditions}
T_{\ell} = 0 \ , \ \qquad \ \ell = 0, 1, \ldots, q-1 \ , \qquad \mbox{and} \qquad T_{q} = -1 \ . \\
\end{equation}
We have $q-1$ gaps, and thus $q-1$ \textit{gap conditions}:

\begin{equation}\label{gap conditions}
\Re \int_{b_{j}}^{a_{j+1}} h(s) R^{1/2}(s) \dd s = 0 \ , \ \ j = 1, \ldots, q-1 \ .
\end{equation}
Since the equilibrium
measure is positive along the support, we immediately find the following $q-1$ real conditions 
\begin{equation}\label{cut conditions}
\Re \int_{a_{j}}^{b_{j}} h(s) R^{1/2}_{+}(s) \dd s = 0 \ , \ \ j = 1, \ldots, q-1 \ .
\end{equation}
Notice that the condition on the last cut $$\Re \int_{a_{q}}^{b_{q}} h(s) R^{1/2}_{+}(s) \dd s = 0 \ ,$$
is a consequence of the $q-1$ conditions in \eqref{cut conditions} and should not be considered as an extra requirement.

Being in the $q$-cut case, we have to determine $2q$ endpoints and thus $4q$ real unknowns $\Re a_1, \Im a_1$, $\Re a_2, \Im a_2, $ $ \cdots $ $\Re b_q, \Im b_q$. These unknowns are determined by the $4q$ real conditions given by \eqref{moment conditions}, \eqref{gap conditions}, and \eqref{cut conditions}.

Let $\mathscr{F}$ be the vector-valued function, whose $4q$ entries are defined as\begin{eqnarray}\label{4q eqns}
&&\mathscr{F}_{2\ell} = \Re T_{\ell}+  \delta_{\ell q}, \ \ \ \ \mathscr{F}_{2\ell+1} = \Im T_{\ell}, \qquad \ell = 0, \ldots,q \ , \\
&&
\mathscr{F}_{2q+1+j}=\Re \int_{a_{j}}^{b_{j}} h(s) R^{1/2}_{+}(s) \dd s, \qquad \ \ j = 1, \ldots, q-1 \ , \\
&&
\mathscr{F}_{3q+j} = \Re \int_{b_{j}}^{a_{j+1}} h(s) R^{1/2}(s) \dd s , \qquad \ \  j = 1, \ldots, q-1 \ .
\end{eqnarray}
We express the equations \eqref{moment conditions}, \eqref{gap conditions}, and \eqref{cut conditions} for determining the branch points as \begin{equation}\label{endpoint equations} \mathscr{F} = 0. \end{equation} 

From the requirement \eqref{em12}, and equation \eqref{g' V' h sqrt R}, in particular, we know that

\begin{equation}
	\deg V - 1= \deg h + \frac{\deg R}{2},
\end{equation}
therefore, recalling \eqref{em1},  \eqref{em23}, and \eqref{em24} we obtain
\begin{equation}\label{r p q}
	r= 2p - 1 - q.
\end{equation}
Since $h$ is a polynomial, we obtain the following bound on the number of cuts 
\begin{equation}
	q \leq 2p-1.
\end{equation}

\begin{definition}\label{Def q cut}
		The regular $q$-cut regime which is denoted by
		$\mathcal{O}_{q}$ is a subset in the phase space  $\C^{2p-1}$ which is defined as the collection of all $\boldsymbol{t}\equiv (t_1, \cdots, t_{2p-1}) \in \C^{2p-1}$ such that the points $a_j\left(\boldsymbol{t}\right) , b_j\left(\boldsymbol{t}\right)$, with $j=1,\cdots, q$ and, $\ell=1, \cdots, 2p-1-q$ as solutions of \eqref{endpoint equations} are all distinct and
		\begin{enumerate}
			\item The set $\mathscr{J}^{(q)}_{\boldsymbol{t}}$ of all points $z$ satisfying \begin{equation*}\label{level set 3}
			\Re \left[ \eta_q(z;\boldsymbol{t}) \right]=0,
			\end{equation*}
			contains a single Jordan arc connecting $a_j(\boldsymbol{t})$ to $b_j(\boldsymbol{t})$, for each $j=1,\cdots, q$.
			\item The points $z_\ell\left(\boldsymbol{t}\right)$ , $\ell=1, \cdots, 2p-1-q$, do not lie on $J^{(q)}_{\boldsymbol{t}} := \bigcup_{j=1}^{q} \Ga_{\boldsymbol{t}}\left[a_j\left(\boldsymbol{t}\right), b_{j}\left(\boldsymbol{t}\right)\right]$.
			\item There exists a complementary arc $\Ga_{\boldsymbol{t}}(b_q\left(\boldsymbol{t}\right), +\infty)$ which lies entirely in the component of the set \begin{equation*}
			\left\{ z : \Re \left[ \eta_q(z;\boldsymbol{t}) \right]<0 \right\},
			\end{equation*} which encompasses $(M_1\left(\boldsymbol{t}\right),+\infty)$ for some $M_1\left(\boldsymbol{t}\right)>0$.
						\item There exists a complementary arc $\Ga_{\boldsymbol{t}}(-\infty, a_1\left(\boldsymbol{t}\right))$ which lies entirely in the component of the set \begin{equation*}
			\left\{ z : \Re \left[ \eta_q(z;\boldsymbol{t}) \right]<0 \right\},
			\end{equation*} which encompasses $( - \infty , -M_2\left(\boldsymbol{t}\right))$ for some $M_2\left(\boldsymbol{t}\right)>0$.
			\item There exists a complementary arc $\Ga_{\boldsymbol{t}}\left(b_j\left(\boldsymbol{t}\right), a_{j+1}\left(\boldsymbol{t}\right)\right)$, for each $j=1,\cdots,q-1$ which lies entirely in the component of the set \begin{equation*}
			\left\{ z : \Re \left[ \eta_q(z;\boldsymbol{t}) \right]<0 \right\}.
			\end{equation*} 
		\end{enumerate}
\end{definition}

\subsection{Structure of the Critical Graph} In this subsection we show basic structural facts about the critical graph $\mathscr{J}^{(q)}_{\boldsymbol{t}}$  for a regular $q$-cut $\boldsymbol{t}$. Recalling \eqref{eta def} we notice that as $z \to \infty$ we have 

\begin{equation}
	\eta_q(z;\boldsymbol{t}) = - \frac{z^{r+q+1}}{r+q+1} \left(1+O(z^{-1})\right) = - \frac{z^{2p}}{2p} \left(1+O(z^{-1})\right),
\end{equation}
where we have used \eqref{r p q}. Therefore the components of $\mathscr{J}^{(q)}_{\boldsymbol{t}}$ near $\infty$ must approach the $4p$ distinct angles $\theta$ 

\begin{equation}\label{angles}
	\theta=\frac{\pi}{4p} + \frac{k \pi}{2p}, \qquad k = 0, 1, \cdots, 4p-1, 
\end{equation}
satisfying $\cos(2p\theta)=0,$ where we have parameterized $z$ in the polar form $ R e^{\ii \theta}$. Moreover, at each endpoint there are three critical trajectories of the quadratic differential $Q(z)\dd z^2$ making angles of $2\pi/3$ at the critical point. To see this, let $\al$ denote either $a_{j}$ or $b_{j}$, $j=1,\cdots,q$. We have 

\[ 	\eta_q(z) = - \int_{b_q}^{\al} h(s)\sqrt{R(s)} \dd s - \int_{\al}^{z} h(s)\sqrt{R(s)} \dd s. \]

Notice the first term on the right hand side is an imaginary number, which can be seen if we break it up into integrals over cuts and gaps and using the endpoint conditions \eqref{gap conditions} and \eqref{cut conditions}. The integrand of the second integral on the right hand side is $O\left((s-\al)^{1/2}\right)$, and thus \[ \Re \eta(z) = O\left((z-\al)^{3/2}\right), \qquad \mbox{as} \qquad z \to \al, \qquad \ell=1,\cdots,q. \]

This ensures that there are $3$ local trajectories emanating from $\al \in \{a_{j},{b_j}\}^{q}_{j=1}$ as solutions of $\Re \eta(z)=0$. Out of these $3 \times 2q$ local critical trajectories, $2q$ of them make the $q$ cuts, and thus we need to determine the destinations of the remaining  $4q$ local critical trajectories. Having solutions in the $4p$ directions given in \eqref{angles} near infinity guides us to investigate if all or some of the $4q$ local critical trajectories can terminate at infinity along one of the $4p$ angles in \eqref{angles}. We have three cases

\begin{enumerate}
	\item $p>q$. This means that the remaining $4q$ local trajectories are not enough to exhaust all $4p$ angles given in \eqref{angles} and thus  $\mathscr{J}^{(q)}_{\boldsymbol{t}}$ must also be constituted from $2(p-q)$ humps to correspond to the \textit{unoccupied} $4(p-q)$ directions at infinity.
	\item $p=q$. in this case $\mathscr{J}^{(q)}_{\boldsymbol{t}}$ does not have any humps, since the remaining $4q$ local trajectories are enough to exhaust all $4p$ angles given in \eqref{angles}.
	\item $p<q$. This means that there are not enough destinations for $4(q-p)$ of the remaining $4q$ local trajectories, and thus the only possibility is that we have $2(q-p)$ connections $[b_{k_m},a_{k_{m}+1}] \subset \mathscr{J}^{(q)}_{\boldsymbol{t}}$ for some index set $$\{k_{1}, \cdots k_{2(q-p)}\} \subseteq \{ 1, \cdots , q-1 \}$$
\end{enumerate}

\begin{remark}
\normalfont	It is clear that all three cases above are realizable for the quartic potential ($p=2$) considered in \cite{BGM}, when we can have $q=1,2,$ and $3$.
\end{remark}

\begin{comment}

\begin{remark}
	Discuss a case where we have the third case above, but the number of gaps are more than the number of "connections" in the gaps, in other words when the index set \[ \{k_{1}, \cdots k_{2(q-p)}\} \]
	
	is a proper subset of $$  \{ 1, \cdots , q-1 \}. $$

\end{remark}
\end{comment}

\begin{remark}\label{humps in K} \normalfont
	If there are $m \leq r$ points $\{z_{k_1}, \cdots ,  z_{k_m}\} \subset \{ z_1, \cdots, z_r \}$ which belong to an unbounded geodesic polygon $K$ with the finite vertex at an endpoint (and the other "vertex" is at infinity), then the separation of the angles between the two edges at $\infty$ is $$\theta_{\infty}=\frac{(2m+1)\pi}{q+r+1} = \frac{(2m+1)\pi}{2p},$$
	and therefore $K$ hosts $m$ "humps" (branches of $\mathscr{J}^{(q)}_{\boldsymbol{t}}$ which start at $\infty$ at one of the angles in \eqref{angles} and also end at $\infty$ at a \textit{consecutive angle}, given again by \eqref{angles}, say at $3\pi/4p$ and $5\pi/4p$.) This is a consequence of the Teichm{\" u}ller's lemma applied to the polygon $K$.
\end{remark}

\section{Solvability of end point equations in a neighborhood of a regular $q-$cut point. Proof of Theorem \ref{solvability thm}}\label{SecSolvability}

In this section we want to prove that the equations uniquely determining the end-points are solvable in a neighborhood of a regular $q$-cut point. In this section we denote\[ \boldsymbol{t} \equiv \left(\Re t_1, \Im t_1, \cdots , \Re t_{2p-1} , \Im t_{2p-1}\right), \]
and 
\[ \boldsymbol{x} \equiv \left(\Re a_1, \Im a_1, \ldots, \Re a_{q}, \Im a_q, \Re b_1, \Im b_1, \ldots, \Re b_{q}, \Im b_q\right). \]
However when we refer to Definition \ref{Def q cut}, by $\boldsymbol{t}$ we denote the complex vector $ (t_1, \cdots, t_{2p-1}) \in \C^{2p-1}$.  We can think of $\mathscr{F}$ as a function of $4q$ real variables in the space $$X :=  \left\{ \left(\Re a_1, \Im a_1, \ldots, \Re a_{q}, \Im a_q, \Re b_1, \Im b_1, \ldots, \Re b_{q}, \Im b_q\right) \ : \ a_j,b_j \in \C, \ j=1,\ldots q
\right\} \ ,$$ and parameters in the space $$V :=  \left\{ \left(\Re t_1, \Im t_1, \cdots , \Re t_{2p-1} , \Im t_{2p-1}\right) \ : \ t_j \in \C, \ j=1,\cdots 2p-1 \right\} \ ,$$ recalling \eqref{em1}\footnote{Notice that the integrand $h(s)R^{1/2}(s)$ only depends on vectors in $X\times V$ due to \eqref{g' V' h sqrt R}, \eqref{g' series}, and \eqref{moments}.}. That is
\begin{eqnarray} \begin{comment}
\mathscr{F} = \mathscr{F}( {\bf A}, {\bf B}) \ . \qquad
\end{comment}
 \mathscr{F} : X \times V \to \R^{4q} \ ,
\end{eqnarray}
or
\begin{eqnarray}
\mathscr{F} : \R^{4q+4p-2} \to \R^{4q} \ .
\end{eqnarray}

Notice that the objects $T_r$, $\int_{b_{m}}^{a_{m+1}} h(s) R^{1/2}(s) \dd s$, and $\int_{a_{i}}^{b_{i}} h(s) R^{1/2}_{+}(s) \dd s$ are complex-analytic with respect to $a_j, b_j,$ and $t_k$, for $0 \leq r \leq q $, $1 \leq m \leq q-1 $, $1 \leq i \leq q$, $1 \leq j \leq q $, and $1 \leq k \leq 2p-1 $.  Here we have used the fact that we know the explicit dependence of $h(s;\boldsymbol{t})$ on  $a_j, b_j,$ and $t_k$ which can be seen as follows: recall from \eqref{g' V' h sqrt R} that 
\begin{equation}
h(z;\boldsymbol{t}) = \frac{1}{R^{1/2}(z;\boldsymbol{t})} \left( V'(z;\boldsymbol{t})- 2 g'(z;\boldsymbol{t}) \right).
\end{equation}
Combining this with \eqref{g' series} we obtain
\begin{equation}\label{h in terms of V and R}
h(z;\boldsymbol{t}) =  -\frac{1}{ \pi \ii } \int_{J} \frac{V'(s;\boldsymbol{t})}{R^{1/2}_{+}(s;\boldsymbol{t})} \frac{\dd s}{s-z} + \frac{V'(z;\boldsymbol{t})}{ R^{1/2}(z;\boldsymbol{t})} ,
\end{equation}
and thus,
\begin{equation}\label{h h h}
h(z;\boldsymbol{t}) =	-\frac{1}{2 \pi \ii} \oint_{\ga^*} \frac{V'(s;\boldsymbol{t})}{R^{1/2}(s;\boldsymbol{t})} \frac{\dd s}{s-z} \ ,
\end{equation}
where $\ga^*$ is a negatively oriented contour which encircles both the support set $J$ and the point $z$.

This means that the functions $\mathscr{F}_{\ell}$, $1\leq \ell \leq 4q$ are all real-analytic functions of $\Re a_j, \Im a_j, \Re b_j, \Im b_j, \Re t_k, \Im t_k$ for  $1 \leq j \leq q $, and $1 \leq k \leq 2p-1 $. This allows us to use the\textit{ real-analytic implicit function theorem}\footnote{For the real-analytic version of the implicit function theorem see, e.g. Theorem 2.3.5 of \cite{KrantzParks}, and for the uniqueness of the map $\phi$, see e.g. Theorem 9.2 of \cite{MunkresAnalysisonManifolds}.}.

\begin{comment}
	\footnote{\red{It is important to see what components of \textit{regularity} are actually needed for this proof? We do use the property that the zeros of $h$ are distinct from zeros of $R$, in \eqref{h neq 0 at endpoints}}}
\end{comment}

We show that if we are in the regular situation, then the Jacobian of the mapping $ \mathscr{F}$ with respect to the parameters in $X$ is nonzero. So, if for some $(\boldsymbol{x}^*,\boldsymbol{t}^*) \in X \times V$,  we have $\mathscr{F}(\boldsymbol{x}^*,\boldsymbol{t}^*) = 0$ and if 
\begin{comment}
content...\mathcal{J} (x^*,v^*) \equiv  \det \boldsymbol{\mathcal{J}} (x^*,v^*) \equiv
\end{comment} 

\begin{equation}\label{Jac}
\left.\det \begin{pmatrix}
\frac{\partial \mathscr{F}_1}{\partial \Re a_1} & \frac{\partial \mathscr{F}_1}{\partial \Im a_1} & \cdots & \frac{\partial \mathscr{F}_{1}}{\partial \Re b_q} & \frac{\partial \mathscr{F}_1}{\partial \Im b_q} \\
\vdots & \vdots & \cdots & \vdots & \vdots \\
\frac{\partial \mathscr{F}_{4q}}{\partial \Re a_1} & \frac{\partial \mathscr{F}_{4q}}{\partial \Im a_1} & \cdots & \frac{\partial \mathscr{F}_{4q}}{\partial \Re b_q} & \frac{\partial \mathscr{F}_{4q}}{\partial \Im b_q} \\
\end{pmatrix} \right\vert_{(\boldsymbol{x}^*,\boldsymbol{t}^*)} \neq 0 \ ,
\end{equation}
then the real-analytic implicit function theorem ensures that there exists a neighborhood $\Om_1$ of $$\boldsymbol{t}^* \equiv \left(\Re t^*_1, \Im t^*_1, \cdots , \Re t^*_{2p-1} , \Im t^*_{2p-1}\right) \in \R^{4p-2}, $$ a neighborhood $\Om_2$ of $$\boldsymbol{x}^* \equiv  \left(\Re a^*_1, \Im a^*_1, \ldots, \Re a^*_{q}, \Im a^*_q, \Re b^*_1, \Im b^*_1, \ldots, \Re b^*_{q}, \Im b^*_q\right)\in \R^{4q},$$ and a unique real-analytic mapping $\varphi: \Om_1 \to \Om_2$ such that $\varphi(\boldsymbol{t}^*)=\boldsymbol{x}^*$, and $\mathscr{F}(\varphi(\boldsymbol{t}),\boldsymbol{t}) = 0$ for all $\boldsymbol{t} \in \Om_1$. 

Due to the continuity of $\varphi$, and the fact that $\boldsymbol{t}^*$ is a regular $q$-cut point  (so all $a_j(\boldsymbol{t}^*), b_j(\boldsymbol{t}^*)$ are distinct)  we can find a possibly smaller neighborhood $\Om_0 \subset \Om_1$ so that for each $\boldsymbol{t} \in \Om_0$ all end-points $a_j(\boldsymbol{t}), b_j(\boldsymbol{t})$, $j=1,\cdots,q$, are distinct.

\begin{comment}
	\blue{Due to continuity of $\varphi$, for each $j=1, \cdots, q$, there also exist the following $2q$ neighborhoods of $\boldsymbol{t}^*$
	\begin{itemize}
	\item $\hat{\Om}_j$  in which $a_j(\boldsymbol{t}) \neq b_j(\boldsymbol{t})$, and
	\item $\tilde{\Om}_j$ in which $b_j(\boldsymbol{t}) \neq a_{j+1}(\boldsymbol{t})$\footnote{and set $\tilde{\Om}_q \equiv \emptyset$.}.
	\end{itemize}
	
	So the desired neighborhood $\Om$ in the statement of the theorem is
	\begin{equation}
	\Om = \Om_1 \cap \left( \bigcap^{q}_{j=1} \left(\hat{\Om}_j \cap \tilde{\Om}_j \right) \right).
	\end{equation} }
\end{comment}

So it only remains to prove that the Jacobian is nonzero at a regular $q$-cut point. We assume the Jacobian is zero at such a point, and aim for a contradiction.   Starting with this assumption, we know that there is $\boldsymbol{0} \neq \tilde{\boldsymbol{x}} \in X$ in the nullspace of the Jacobian matrix. Using $\boldsymbol{x}^*$ and $\tilde{\boldsymbol{x}}$ we define the following $1$-parameter family
\begin{equation}
	\boldsymbol{x}(\tau):= \boldsymbol{x}^* + \tau \tilde{\boldsymbol{x}} \ , \qquad \tau \in \R \ .
\end{equation}
We obviously have 
\begin{equation}\label{nonzero derivative}
\left. \frac{\dd}{\dd \tau} \boldsymbol{x}(\tau) \right|_{\tau=0} = \tilde{\boldsymbol{x}} \neq \boldsymbol{0} \ .
\end{equation}
For non-zero values of $\tau$, $\boldsymbol{x}(\tau)$ may not satisfy the end-point equations \eqref{endpoint equations}, but we can still think of the entries of $\boldsymbol{x}(\tau)$ as defining "end-points". More precisely, we define the points $a_j(\tau)$ and $b_j(\tau)$, as  $\Re a_j(\tau) = \boldsymbol{x}_{2j-1}(\tau)$, $\Im a_j(\tau) = \boldsymbol{x}_{2j}(\tau)$, $\Re b_j(\tau) = \boldsymbol{x}_{2q+2j-1}(\tau)$, $\Im b_j(\tau) = \boldsymbol{x}_{2q+2j}(\tau)$, $j=1,\cdots,q$. Now, using $a_j(\tau)$ and $b_j(\tau)$ as defined above, we define the $\tau$-dependent objects $R(z;\tau)$, $g'(z;\tau)$ and $T_\ell(\tau)$ using \eqref{em24}, \eqref{g' series} and \eqref{moments}. Now \eqref{g' V' h sqrt R} gives an expression for $h(z;\tau) \sqrt{R(z;\tau)}$.
\begin{comment}
	Notice that since $g(z;\tau)$ as defined here does not necessarily tend to zero as $z \to \infty$, the difference of its boundary values is not given by the Cauchy transform of $h(z;\tau) \sqrt{R(z;\tau)}$, This only happens if \eqref{moment conditions} all hold for the deformed points. 
\end{comment}
% and 
%\begin{eqnarray}
%\left.
%\frac{\dd}{\dd \tau} \mathscr{F}(\boldsymbol{x}(\tau) , \boldsymbol{v}^*)  \right|_{\tau = 0} = \boldsymbol{0} \ ,
%\end{eqnarray}
%by construction. 
 We emphasize that for non-zero $\tau$, these objects may not correspond to an equilibrium measure for some potential $V(\tau)$. 
	
	Below, we drop the dependence on $\tau$ in the notations to simplify our presentation. Notice that
\begin{equation}
	\label{eq:Fpartials1}
	\frac{\dd}{\dd \tau} T_{\ell}=  \sum_{j=1}^{q}  \left( \frac{\dd a_{j}(\tau)}{\dd \tau}  \frac{\partial}{\partial a_{j}} + \frac{\dd b_{j}(\tau)}{\dd \tau}  \frac{\partial}{\partial b_{j}} \right) T_{\ell} \ ,
\end{equation}

\begin{equation}\label{eq:Fpartials2}
	\frac{\dd}{\dd \tau} \Re \int_{a_{j}}^{b_{j}} h(s) R^{1/2}_{+}(s) \dd s = 
	\Re \int_{a_{j}}^{b_{j}} \sum_{j=1}^{q}  \left( \frac{\dd a_{j}(\tau)}{\dd \tau}  \frac{\partial}{\partial a_{j}} + \frac{\dd b_{j}(\tau)}{\dd \tau}  \frac{\partial}{\partial b_{j}} \right) \ \ h(s) R^{1/2}_{+}(s) \dd s \ ,
\end{equation}
and
\begin{equation}
	\label{eq:Fpartials3}
	\frac{\dd}{\dd \tau} \Re \int_{b_{j}}^{a_{j+1}} h(s) R^{1/2}(s) \dd s = \Re \int_{b_{j}}^{a_{j+1}}\sum_{j=1}^{q}  \left( \frac{\dd a_{j}(\tau)}{\dd \tau}  \frac{\partial}{\partial a_{j}} + \frac{\dd b_{j}(\tau)}{\dd \tau}  \frac{\partial}{\partial b_{j}} \right) \ \ h(s) R^{1/2}(s) \dd s \ .
\end{equation}
We let $\alpha$ represent an arbitrary branch point $a_{j}$ or $b_{j}$.  From \eqref{moments} we have the identity
\begin{eqnarray}
T_{\ell} - \alpha T_{\ell-1} = \int_{J} \frac{ V'(s)}{R^{1/2}_{+}(s)} s^{\ell-1} ( s - \alpha) \dd s \ , \qquad \ell \in \N \ .
\end{eqnarray}
Differentiating with respect to $\al$ yields
\begin{eqnarray}
\frac{\partial}{\partial \alpha} T_{\ell} - \alpha \frac{\partial}{\partial \alpha}T_{\ell-1} - T_{\ell-1} = \frac{-1}{2}\int_{J} \frac{ V'(s)}{R^{1/2}_{+}(s)} s^{\ell-1}  \dd s = \frac{-1}{2} T_{\ell-1} \  ,
\end{eqnarray}
which implies 
\begin{eqnarray}
\frac{\partial}{\partial \alpha} T_{\ell} - \alpha \frac{\partial}{\partial \alpha}T_{\ell-1}  =  \frac{1}{2} T_{\ell-1} \ .
\end{eqnarray}
In view of \eqref{moment conditions} for $1 \leq \ell \leq q$, when $\tau=0$ we actually have
\begin{eqnarray}
\frac{\partial}{\partial \alpha} T_{\ell} = \alpha \frac{\partial}{\partial \alpha}T_{\ell-1}  = \alpha^{\ell} \frac{\partial}{\partial \alpha} T_{0} \ .
\end{eqnarray}

\begin{lemma}
	We have
			\begin{eqnarray}\label{partial al g'}
		\frac{\partial}{\partial \alpha} g'(z) = \frac{- R^{1/2}(z)}{ 2 \pi i(z - \alpha)} \frac{\partial}{\partial \alpha} T_{0} \ .
		\end{eqnarray}
\end{lemma}

\begin{proof}
	Let us rewite \eqref{g' series} as 
	
	\begin{equation}
	\begin{split}
		    g'(z) & = \frac{R^{1/2}(z)}{2 \pi i } \int_{J} \frac{V'(s)(s-\al)}{R^{1/2}_{+}(s)} \frac{\dd s}{(s-z)(s-\al)} \\ & =- \frac{R^{1/2}(z)}{2 \pi i (z-\al) } \int_{J} \frac{V'(s)}{R^{1/2}_{+}(s)} \dd s + \frac{R^{1/2}(z)}{2 \pi i (z-\al) } \int_{J} \frac{V'(s)(s-\al)}{R^{1/2}_{+}(s)} \frac{\dd s}{s-z} \\ & =  \frac{R^{1/2}(z)}{2 \pi i (z-\al) } \left( -T_0 +  \int_{J} \frac{V'(s)(s-\al)}{R^{1/2}_{+}(s)} \frac{\dd s}{s-z}\right) \ .
	\end{split}
	\end{equation}
	The advantage of this formula is that the differentiation with respect to $\al$ can be pushed through the integral, as the integrand  vanishes at $\al$. After taking the derivative with respect to $\al$ and straight-forward simplifications we obtain \eqref{partial al g'}
\end{proof}

Returning to (\ref{eq:Fpartials1})-(\ref{eq:Fpartials3}), when $\tau=0$ we have
\begin{eqnarray}\label{alg eqns}
\frac{\dd}{\dd \tau} T_{\ell} = 
\sum_{j=1}^{q}  \left( \frac{\dd a_{j}(\tau)}{\dd \tau}  \frac{\partial T_{0}}{\partial a_{j}} a_{j}^{\ell} + \frac{\dd b_{j}(\tau)}{\dd \tau}  \frac{\partial T_{0}}{\partial b_{j}}  b_{j}^{\ell} \right) \ , \qquad  \ell = 0 , \ldots q \ ,
\end{eqnarray}
\begin{eqnarray}
\label{eq:RealPart1}
\frac{\dd}{\dd \tau}\Re \int_{a_{j}}^{b_{j}} h(s) R^{1/2}_{+}(s) \dd s
= \Re \frac{1}{\pi \ii}\int_{a_{j}}^{b_{j}} \sum_{j=1}^{q}  \left( \frac{\dd a_{j}(\tau)}{\dd \tau}  \frac{\partial T_{0}}{\partial a_{j}} \frac{1}{s-a_{j}} + \frac{\dd b_{j}(\tau)}{\dd \tau}  \frac{\partial T_{0}}{\partial b_{j}}\frac{1}{s-b_{j}} \right) \left( R^{1/2}_{+}(s)\right) \dd s \ ,
\end{eqnarray}
\begin{eqnarray}
 \label{eq:RealPart2}
\frac{d}{d \tau}\Re \int_{b_{j}}^{a_{j+1}} h(s) R^{1/2}_{+}(s) \dd s
= \Re \frac{1}{\pi \ii}\int_{b_{j}}^{a_{j+1}} \sum_{j=1}^{q}  \left( \frac{\dd a_{j}(\tau)}{\dd \tau}  \frac{\partial T_{0}}{\partial a_{j}} \frac{1}{s-a_{j}} + \frac{\dd b_{j}(\tau)}{\dd \tau}  \frac{\partial T_{0}}{\partial b_{j}}\frac{1}{s-b_{j}} \right) \left(R^{1/2}(s)\right) \dd s \ ,
\end{eqnarray}
where in deriving the last two equations we have used \eqref{g' V' h sqrt R} and \eqref{partial al g'}. Consider the $2q$-vector 
\begin{eqnarray}\label{W}
 \boldsymbol{W} := \di \left . \left( \begin{array}{c}
 \frac{\partial T_{0}}{\partial a_{1}}
\frac{\dd a_{1}}{\dd \tau} \\
\vdots\\
 \frac{\partial T_{0}}{\partial a_{q}}
\frac{\dd a_{q}}{\dd \tau} \\
 \frac{\partial T_{0}}{\partial b_{1}}
\frac{\dd b_{1}}{\dd \tau} \\
\vdots\\
 \frac{\partial T_{0}}{\partial b_{q}}
\frac{\dd b_{q}}{\dd \tau} \\
\end{array}\right) \right|_{\tau=0}
\end{eqnarray}
that, in view of \eqref{alg eqns}, satisfies the $q+1$ equations 
\begin{eqnarray}
\label{eq:VdMConds}
\left( \begin{array}{cccccc}
    1 & \cdots & 1 & 1 & \cdots & 1 \\
 a_{1} & \cdots & a_{q} & b_{1} & \cdots & b_{q} \\   
 \vdots & \ddots & \vdots & \vdots & \ddots & \vdots \\
 a_{1}^{q} & \cdots & a_{q}^{q} & b_{1}^{q} & \cdots & b_{q}^{q} \\  
\end{array} \right) \boldsymbol{W} = \boldsymbol{0} \ ,
\end{eqnarray}
where $a_j, \ b_j$, $j=1,\cdots, q$ are all evaluated at $\tau=0$, in other words they are the actual endpoints corresponding to the solution $(\boldsymbol{x}^*,\boldsymbol{v}^*)$.
Furthermore, the integrand in (\ref{eq:RealPart1}) and (\ref{eq:RealPart2}) can be described as follows.  We first define
\begin{eqnarray}
 \boldsymbol{\rho}(s) := \left( \begin{array}{c}
  \frac{1}{s- a_{1}}    \\
\vdots \\
\frac{1}{s-a_{q}} \\
\frac{1}{s-b_{1}} \\
\vdots \\
\frac{1}{s-b_{q}}\\
\end{array} \right) \ ,
\end{eqnarray}
and then 
\begin{eqnarray}\label{BB}
B(s) := \frac{1}{\pi \ii} \left( \boldsymbol{W}^T  \boldsymbol{\rho}(s) \right) \ R^{1/2}(s)  \ ,
\end{eqnarray}
Where $(\cdot)^T$ denotes the transpose and all objects are evaluated at $\tau=0$. In other words, 
\begin{eqnarray}\label{B}
B(s) = \frac{1}{\pi \ii} \sum_{j=1}^{q}  \left( \frac{\dd a_{j}(\tau)}{\dd \tau}  \frac{\partial T_{0}}{\partial a_{j}} \frac{1}{s-a_{j}} + \frac{\dd b_{j}(\tau)}{\dd \tau}  \frac{\partial T_{0}}{\partial b_{j}}\frac{1}{s-b_{j}} \right) \ R^{1/2}(s) \ .
\end{eqnarray}
This function, in view of \eqref{eq:RealPart1} and \eqref{eq:RealPart2} satisfies
\begin{eqnarray}\label{Real Part Condition on Cuts}
\Re \int_{a_{j}}^{b_{j}} B_{+}(s) \dd s = 0 \ , \qquad j = 1, \ldots, q-1 \ , \\
\label{Real Part Condition on Gaps}\Re \int_{b_{j}}^{a_{j+1}} B(s) \dd s = 0 \ , \qquad j = 1, \ldots, q-1 \ .
\end{eqnarray}

Now, using \eqref{B} and expanding $(s-\al)^{-1}$ for large $s$ and switching the order of summations we obtain \begin{equation}
\frac{B(s)}{R^{1/2}(s)} = \frac{1}{\pi \ii s} \sum_{\ell=0}^{\infty} \frac{1}{s^{\ell}} \sum_{j=1}^{q}  \left( \frac{\dd a_{j}(\tau)}{\dd \tau}  \frac{\partial T_{0}}{\partial a_{j}} a_{j}^{\ell} + \frac{\dd b_{j}(\tau)}{\dd \tau}  \frac{\partial T_{0}}{\partial b_{j}}b_{j}^{\ell} \right). 
\end{equation}
So, because of (\ref{eq:VdMConds}), and recalling \eqref{em24} we observe that the behavior of $B(s)$ for $s$ large is given by
\begin{eqnarray}\label{B at inf}
B(s) = O\left( \frac{1}{s^{2}} \right) \ .
\end{eqnarray}
Therefore $B$ can be expressed as 
\begin{eqnarray}
B(s) = \frac{Q(s)}{R^{1/2}(s)}\ , 
\end{eqnarray}
where $Q$ is a polynomial of degree at most $q-2$.

Next, we show that $B$ is identically zero. To prove this, we show that the following integral is $0$.
\begin{eqnarray}
\label{eq:integral}
\iint_{\mathbb{C}} B(z) \overline{B(z)} \dd A \ .
\end{eqnarray}

\begin{lemma}
	Let $C$ be a positively oriented, piecewise smooth, simple closed curve in the plane, and let $D$ be the region bounded by $C$. Let $f(x+\ii y)\equiv u(x,y)+\ii v(x,y)$ be analytic in $D$. We have
	\begin{equation}
		\iint_{D} \frac{\partial}{\partial z} f(z) \ \dd A = \frac{\ii}{2} \oint_C f(z) \ovl{\dd z} \ ,
	\end{equation} 
	where Integration with respect to $\overline{\dd z}$ means:  parametrize the contour of integration via $z=z(t)$, then $$\int f(z) \ \overline{\dd z} = \int_{t_{0}}^{t_{1}} f(z(t)) \overline{z'(t)} \ \dd t \ .$$
\end{lemma}
\begin{proof}
	We have
	\begin{equation}\label{16 1}
		\iint_{D} \frac{\partial}{\partial z} f(z) \ \dd A = \iint_{D} \frac{1}{2} \left(\frac{\partial}{\partial x} - \ii \frac{\partial}{\partial y} \right) \left(u(x,y) + \ii v(x,y)\right) \ \dd A = \frac{1}{2} \iint_{D}  \left( (u_x+v_y) + \ii (v_x-u_y) \right) \ \dd A \ .
	\end{equation}
	Now we apply the Green's theorem for the vector fields $\boldsymbol{F}_1(x,y)=(-v(x,y),u(x,y))$, and $\boldsymbol{F}_2(x,y)=(u(x,y),v(x,y))$. So
	
	\[ \iint_{D} \frac{\partial}{\partial z} f(z) \ \dd A = \frac{1}{2} \oint_C (-v(x,y),u(x,y))\cdot(x'(t),y'(t)) \ \dd t + \frac{\ii}{2} \oint_C (u(x,y),v(x,y))\cdot(x'(t),y'(t)) \ \dd t, \]
	
	where $(x(t),y(t))$ is the parametrization of the curve $C$. We therefore have
	
	\begin{equation}
		\begin{split}
		\iint_{D} \frac{\partial}{\partial z} f(z) \ \dd A & = \frac{1}{2} \oint_C \bigg( x'(t)[-v(x,y)+\ii u(x,y)] + y'(t)[u(x,y)+\ii v(x,y)] \bigg)  \ \dd t, \\ & = \frac{\ii}{2} \oint_C \bigg( u(x,y)+\ii v(x,y) \bigg) \left( x'(t) - \ii y'(t) \right)  \ \dd t, \\ & = \frac{\ii}{2} \oint_C f(z) \ovl{\dd z}.
		\end{split}
	\end{equation}
\end{proof}
Defining 
\begin{eqnarray}
u(z) = \int_{+\infty}^{z}B(s)\dd s \ ,
\end{eqnarray}
we have
\begin{eqnarray}
B(z) \overline{B(z)} = \frac{\partial}{\partial z} \left( u(z) \overline{B(z)} \right) \ , 
\end{eqnarray}
since $\frac{\partial}{\partial z} \overline{B(z)}=0$.
Now, in order to apply Stokes' theorem to the integral (\ref{eq:integral}), we have to apply it in two regions, one \textit{above} the contour of integration $\Gamma$, and one \textit{below} the contour of integration $\Gamma$.

Let $D_r$ be a disk of radius $r$ centered at the origin. The max-min contour $\Ga \in \mathcal{T}$ (see \S \ref{section 2.1}) divides $D_r$ into two parts: $D^{(+)}_r$ \textit{above} $\Ga$, and $D^{(-)}_r$  \textit{below} $\Ga$.  We can write
\begin{equation}
\begin{split}
	\iint_{\mathbb{C}} B(z) \overline{B(z)} \dd A & = \iint_{\mathbb{C}} \frac{\partial}{\partial z} \left( u(z) \overline{B(z)} \right)  \dd A = \lim_{r \to \infty} \iint_{D_r} \frac{\partial}{\partial z} \left( u(z) \overline{B(z)} \right) \dd A \\ & = \lim_{r \to \infty} \iint_{D^{(+)}_r} \frac{\partial}{\partial z} \left( u(z) \overline{B(z)} \right) \dd A + \lim_{r \to \infty} \iint_{D^{(-)}_r} \frac{\partial}{\partial z} \left( u(z) \overline{B(z)} \right) \dd A \\ & 	= \frac{\ii}{2} \lim_{r \to \infty} \oint_{\partial D^{(+)}_r} u(z) \overline{B(z)} \ \ovl{\dd z} + \frac{\ii}{2} \lim_{r \to \infty} \oint_{\partial D^{(-)}_r} u(z) \overline{B(z)} \ \ovl{\dd z} \ , 
\end{split}
\end{equation}
where both $\partial D^{(+)}_r$ and $\partial D^{(-)}_r$ are positively oriented. Therefore, due to \eqref{B at inf}, we find
\begin{eqnarray}\label{stokes}
- 2 i \iint_{\C} B(z) \overline{B(z)} \ \dd A = \int_{\Gamma} \left\{ \left[ u(z) \overline{B(z)}\right]_{+} -  \left[ u(z) \overline{B(z)}\right]_{-} \right\} \overline{\dd z} \ .
\end{eqnarray}
The contour $\Gamma$ is comprised of bands $\Ga[a_{j}, b_{j}]$, gaps $\Ga(b_{j}, a_{j+1})$, and the two semi-infinite contours ($\Ga(-\infty, a_{1})$ from $-\infty$ to $a_{1}$ and $\Ga(b_q, \infty)$ from $b_{q}$ to $+\infty$).  First observe that
\begin{eqnarray}
\int_{b_{q}}^{+\infty} \left\{ \left[ u(z) \overline{B(z)}\right]_{+} -  \left[ u(z) \overline{B(z)}\right]_{-} \right\} \overline{\dd z} = 0 \ ,
\end{eqnarray}
since by definition $u$ and $B$ are continuous across the contour from $b_{q}$ to $+\infty$.

Second, note that for $z$ in the contour from $-\infty$ to $a_{1}$, $B(z)$ is continuous across $\Gamma$, and so is $u(z)$, since 
\begin{eqnarray}
u_{+}(z) - u_{-}(z) = \sum_{j=1}^{q}\int_{a_{j}}^{b_{j}} \left(B_+(s)-B_-(s) \right) \dd s = \sum_{j=1}^{q} \oint_{\ga_{j}} B(s) \dd s = \oint_{\ga^*} B(s) \dd s = 0 \ ,
\end{eqnarray}
due to \eqref{B at inf}, where $\ga_{j}$ is a clockwise contour encircling the cut $\Ga[a_{j},b_{j}]$, and $\ga^*$ is a clockwise contour encircling all cuts. Therefore we also know that
\begin{eqnarray}
\int_{-\infty}^{a_{1}} \left\{ \left[ u(z) \overline{B(z)}\right]_{+} -  \left[ u(z) \overline{B(z)}\right]_{-} \right\} \overline{\dd z} = 0 \ .
\end{eqnarray}
So we must consider 
\begin{eqnarray}
\int_{a_{1}}^{b_{q}} \left\{ \left[ u(z) \overline{B(z)}\right]_{+} -  \left[ u(z) \overline{B(z)}\right]_{-} \right\} \overline{\dd z} \ .
\end{eqnarray}
There are two types of integrals:  those across bands and those across gaps.  

 For $z$ in a band, say  $\Ga(a_j,b_j)$, the quantity $B(z)$ has a jump discontinuity across the contour: $B_{+}(z)=-B_{-}(z)$. So we have
\begin{eqnarray}
\int_{a_{j}}^{b_{j}} \left\{ \left[ u(z) \overline{B(z)}\right]_{+} -  \left[ u(z) \overline{B(z)}\right]_{-} \right\}\overline{\dd z} = \int_{a_{j}}^{b_{j}} \overline{B_{+}(z)} \left( u_{+}(z) + u_{-}(z) 
\right) \overline{\dd z} \ , 
\end{eqnarray} 
and $u_{+} + u_{-}$ is the following constant:
\begin{eqnarray}
u_{+}(z) + u_{-}(z) = -2\sum_{k=j}^{q-1}\int_{b_{k}}^{a_{k+1}} B(s) \dd s - 2 \int_{b_{q}}^{+\infty} B(s) \dd s, \qquad  \mbox{ for all } z \in \Ga(a_j,b_j) \ ,
\end{eqnarray}
and so we find
\begin{eqnarray}
\int_{a_{j}}^{b_{j}} \left\{ \left[ u(z) \overline{B(z)}\right]_{+} -  \left[ u(z) \overline{B(z)}\right]_{-} \right\} \overline{\dd z} = 
\left( 
-2\sum_{k=j}^{q-1}\int_{b_{k}}^{a_{k+1}} B(s) \dd s - 2 \int_{b_{q}}^{+\infty} B(s) \dd s 
\right) \int_{a_{j}}^{b_{j}} \overline{B_+(z)} \ \overline{\dd z} \ .
\end{eqnarray} 
Note that 
\begin{eqnarray}
\int_{a_{j}}^{b_{j}} \overline{B_+(z)} \ \overline{\dd z} = \overline{ \int_{a_{j}}^{b_{j}} B_+(z) \dd z } = - \int_{a_{j}}^{b_{j}} B_+(z) \dd z \ ,
\end{eqnarray}
where the last equality follows from \eqref{Real Part Condition on Cuts}.  Therefore 
\begin{eqnarray}
\label{eq:bandint}
\int_{a_{j}}^{b_{j}} \left\{ \left[ u(z) \overline{B(z)}\right]_{+} -  \left[ u(z) \overline{B(z)}\right]_{-} \right\} \overline{\dd z} = 
\left( 
2\sum_{k=j}^{q-1}\int_{b_{k}}^{a_{k+1}} B(s) \dd s + 2 \int_{b_{q}}^{+\infty} B(s) \dd s 
\right) \int_{a_{j}}^{b_{j}} B_+(z) \dd z \ .
\end{eqnarray}

 For $z$ in a gap, say $\Ga(b_j,a_{j+1})$, the quantity $B(z)$ is continuous across the contour $\Gamma$, therefore
\begin{eqnarray}
\int_{b_{j}}^{a_{j+1}} \left\{ \left[ u(z) \overline{B(z)}\right]_{+} -  \left[ u(z) \overline{B(z)}\right]_{-} \right\} \overline{\dd z} = \int_{b_{j}}^{a_{j+1}} \overline{B(z)} \left(
u_{+}(z) - u_{-}(z)
\right) \overline{\dd z} \ ,
\end{eqnarray}
and $u_{+}(z) - u_{-}(z)$ is the constant:
\begin{eqnarray}
u_{+}(z) - u_{-}(z) = -2 \sum_{k=j+1}^{q} \int_{a_{k}}^{b_{k}} B_+(s) \dd s, \ \qquad \mbox{ for all } z \in \Gamma(b_{j},a_{j+1}) \ .
\end{eqnarray}
And so we find
\begin{eqnarray}
\int_{b_{j}}^{a_{j+1}} \left\{ \left[ u(z) \overline{B(z)}\right]_{+} -  \left[ u(z) \overline{B(z)}\right]_{-} \right\} \overline{\dd z} =
\left( -2 \sum_{k=j+1}^{q} \int_{a_{k}}^{b_{k}} B_+(s) \dd s \right) \int_{b_{j}}^{a_{j+1}} \overline{B(z)} \  \overline{\dd z} \ .
\end{eqnarray}
Note that from \eqref{Real Part Condition on Gaps} we have
\begin{eqnarray}
\int_{b_{j}}^{a_{j+1}} \overline{B(z)} \  \overline{\dd z} = - \int_{b_{j}}^{a_{j+1}} B(z) \dd z \ .
\end{eqnarray}
Therefore
\begin{eqnarray}
\label{eq:gapint}
\int_{b_{j}}^{a_{j+1}} \left\{ \left[ u(z) \overline{B(z)}\right]_{+} -  \left[ u(z) \overline{B(z)}\right]_{-} \right\} \overline{\dd z} =
\left( 2 \sum_{k=j+1}^{q} \int_{a_{k}}^{b_{k}} B_+(s) \dd s \right) \int_{b_{j}}^{a_{j+1}} B(z) \dd z \ .
\end{eqnarray}

 Using (\ref{eq:bandint}) and (\ref{eq:gapint}), we have
 \begin{equation}\label{eq:57}
 \begin{split}
 	& \int_{a_{1}}^{b_{q}}  \left\{ \left[ u(z) \overline{B(z)}\right]_{+} -  \left[ u(z) \overline{B(z)}\right]_{-} \right\} \overline{\dd z} = \\ & \sum_{j=1}^{q} \left( 
 2\sum_{k=j}^{q}\int_{b_{k}}^{a_{k+1}} B(s) \dd s 
 \right) \int_{a_{j}}^{b_{j}} B_+(z)\dd z + \sum_{j=1}^{q-1} \left( 2 \sum_{k=j+1}^{q} \int_{a_{k}}^{b_{k}} B_+(s) \dd s \right) \int_{b_{j}}^{a_{j+1}} B(z) \dd z \ .
 \end{split}
 \end{equation}
Note that in (\ref{eq:57}), $a_{q+1} = + \infty$. Reversing orders of summation in the first term on the r.h.s. of (\ref{eq:57}), we have
\begin{equation}\label{eq:60}
	\begin{split}
	& \int_{a_{1}}^{b_{q}}  \left\{ \left[ u(z) \overline{B(z)}\right]_{+} -  \left[ u(z) \overline{B(z)}\right]_{-} \right\} \overline{\dd z}= \\ & 
	2 \sum_{k=1}^{q}
	\int_{b_{k}}^{a_{k+1}} B(s) \dd s  \sum_{j=1}^{k}
	\int_{a_{j}}^{b_{j}} B_+(z) \dd z + \sum_{j=1}^{q-1} \left( 2 \sum_{k=j+1}^{q} \int_{a_{k}}^{b_{k}} B_+(s) \dd s \right) \int_{b_{j}}^{a_{j+1}} B(z) \dd z \ .
	\end{split}
\end{equation}
Exchanging indices of summation in the first term on the r.h.s. of (\ref{eq:60}), we find
\begin{equation}
	\begin{split}
	& \int_{a_{1}}^{b_{q}}  \left\{ \left[ u(z) \overline{B(z)}\right]_{+} -  \left[ u(z) \overline{B(z)}\right]_{-} \right\} \overline{\dd z}= \\ &
	2 \sum_{j=1}^{q}
	\int_{b_{j}}^{a_{j+1}} B(s) \dd s  \sum_{k=1}^{j}
	\int_{a_{k}}^{b_{k}} B(z)\dd z + \sum_{j=1}^{q-1} \left( 2 \sum_{k=j+1}^{q} \int_{a_{k}}^{b_{k}} B(s) \dd s \right) \int_{b_{j}}^{a_{j+1}} B(z) \dd z = \\ & 2 \sum_{j=1}^{q-1} \left[ 
	\int_{b_{j}}^{a_{j+1}} B(s) \dd s \left(  \sum_{k=1}^{q}
	\int_{a_{k}}^{b_{k}} B(z)\dd z \right) \right] +  \int_{b_{q}}^{+\infty}B(s)\dd s \left(
	\sum_{k=1}^{q}
	\int_{a_{k}}^{b_{k}} B(z)\dd z  \right) = 0 \ .
	\end{split}
\end{equation}

So we have proven that 
\begin{eqnarray}
\iint_{\mathbb{C}} B(z) \overline{B(z)} dA \ = \ 0 \ .
\end{eqnarray}
This of course implies that $B \equiv 0$, and hence by \eqref{BB} we conclude that
\begin{equation}\label{W=0}
	\boldsymbol{W} \equiv 0 \ .
\end{equation}

%This is essentially the proof of Riemann's bilinear identities, brought back to the complex plane from the horrific swampland of Riemann surface theory.

\begin{lemma}
	Let $\al \in \{a_j,b_j \ , \ j =1, \ \cdots \ q \}$, and $T_0$ given by \eqref{moments}. It holds that
\begin{eqnarray}
\frac{\partial T_{0}}{\partial \alpha} \neq 0 \ . 
\end{eqnarray}	
\end{lemma}

\begin{proof}

We have
\begin{eqnarray}
- 2 \pi i \psi(z) = -  h(z) R^{1/2}_{+}(z)= 
 g'(z) - \frac{V'(z)}{2} \ ,
 \end{eqnarray}
 and our assumptions imply that this quantity vanishes like a square root at each branchpoint $\alpha$.  So we know that 
 \begin{eqnarray}\label{h neq 0 at endpoints}
 \lim_{z \to \alpha} \frac{1}{R^{1/2}(z)} \left( g'(z) - \frac{V'(z)}{2} \right)  \neq 0 \ .
 \end{eqnarray}
 Using (\ref{g' series}) we can write
 \begin{eqnarray}\label{contour integral for h}
 \frac{1}{R^{1/2}(z)} \left( g'(z) - \frac{V'(z)}{2} \right) 
 = 
 \frac{1}{4 \pi \ii} \oint_{\ga^*} \frac{V'(s)}{R^{1/2}(s)} \frac{\dd s}{s-z} \ ,
\end{eqnarray}
where $\ga^*$ is a negatively oriented contour which encircles both the support set $J$ and the point $z$. Taking the limit as $z \to \alpha$, and recalling \eqref{h neq 0 at endpoints} we find 
\begin{eqnarray}
\label{eq:ineq}
\frac{1}{4 \pi \ii} \oint_{\ga^*}  \frac{V'(s)}{R^{1/2}(s)} \frac{\dd s}{s-\alpha} \neq 0 \ .
\end{eqnarray}
Recalling the definition \eqref{moments} we can write
\begin{eqnarray}
T_{0} = \int_{J} \frac{V'(s)}{R^{1/2}_+(s)} \dd s = \frac{1}{2} \oint_{\ga^*} \frac{V'(s)}{R^{1/2}(s)}\dd s \ .
\end{eqnarray}
Differentiating with respect to the branchpoint $\alpha$, we find
\begin{eqnarray}
\frac{\partial T_{0}}{\partial \alpha} = \frac{1}{4} \oint_{\ga^*} \frac{V'(s)}{R^{1/2}(s)} \frac{ \dd s}{s - \alpha} \ ,
\end{eqnarray}
which is nonzero because of (\ref{eq:ineq}).
\end{proof}

Recalling \eqref{W}, the above lemma together with \eqref{W=0} imply that 
\begin{equation}
\left. \frac{\dd a_j}{\dd \tau} \right|_{\tau=0} = \left. \frac{\dd b_j}{\dd \tau} \right|_{\tau=0} = 0 \ , \ \qquad \mbox{for all} \ \  j = 1 , 2, \ldots, q. 
\end{equation}

We have thus shown that 
\begin{eqnarray}
\left. \frac{\dd}{\dd \tau} \boldsymbol{x}(\tau) \right|_{\tau=0} = \boldsymbol{0} \ , 
\end{eqnarray}
which contradicts \eqref{nonzero derivative}. This proves that the Jacobian \eqref{Jac} is indeed non-zero. We have thus concluded the proof of Theorem \ref{solvability thm}.

\section{Openness of the regular $q$-cut regime}\label{Sec Openness}

\subsection{Proof of Theorem \ref{continuous deformations of one cut critical graph1}}

\begin{comment}
	\begin{theorem}\label{continuous deformations of one cut critical graph}
	The critical graph $\mathscr{J}^{(q)}_{\boldsymbol{t}} $ deforms continuously with respect to $\boldsymbol{t}$.
	\end{theorem}
	\begin{proof}
\end{comment}

	The end-points $a_j(\boldsymbol{t}), b_j(\boldsymbol{t}) \in \mathscr{J}^{(q)}_{\boldsymbol{t}}$ deform continuously with respect to $\boldsymbol{t}$ as shown in Theorem \ref{solvability thm}, due to nonsingularity of the Jacobian matrix at a regular $q$-cut point. Notice that the critical graph of the quadratic differential $Q(z;\boldsymbol{t})\dd z^2$ is intrinsic to the polynomial $Q$ and does not depend on the particular branch chosen for its natural parameter $\eta_q(z;\boldsymbol{t})$, for example the one chosen in \eqref{eta def}. For the purposes of this proof, for each fixed $\boldsymbol{t}$, unlike our choice in \eqref{eta def}, we choose the branch $\tilde{\eta}_q(z;\boldsymbol{t})$ whose branch cut has no intersections with the critical graph $\mathscr{J}^{(q)}_{\boldsymbol{t}}$ and we can characterize the critical graph of $Q(z;\boldsymbol{t})\dd z^2$ as the totality of solutions to \begin{equation}\label{Main Eq}
\Re \tilde{\eta}_q(z;\boldsymbol{t})=0.
\end{equation} 
	Recalling \eqref{eta def} with the choice of branch discussed above, notice that
		\begin{equation}\label{detadz}
		\frac{\partial \tilde{\eta}_q}{\partial z}(z^*;\boldsymbol{t}^*) = -\prod_{\ell=1}^r \left(z^*-z_\ell(\boldsymbol{t}^*)\right)\left[\prod_{j=1}^q \left(z^*-a_j(\boldsymbol{t}^*)\right)\left(z^*-b_j(\boldsymbol{t}^*)\right)\right]^{1/2} \neq 0,
		\end{equation}
		where $z^* \in \mathscr{J}^{(q)}_{\boldsymbol{t}^*}\setminus\left\{ a_j(\boldsymbol{t}^*), b_j(\boldsymbol{t}^*) \right\}^{q}_{j=1}$ does not lie on the branch cut chosen to define $\tilde{\eta}_q$.	Since $z^*$ is not on the branch cut, there is a small neighborhood of $z^*$ in which $	 \tilde{\eta}_q(z,\boldsymbol{t}^*)$ is analytic. By Cauchy-Riemann equations, from \eqref{detadz} we conclude that at least one of the quantities $ 	\frac{\partial \Re \tilde{\eta}_q}{\partial x}(z^*;\boldsymbol{t}^*)  $ or $ 	\frac{\partial \Re \tilde{\eta}_q}{\partial y}(z^*;\boldsymbol{t}^*)  $ is non zero, $z=x+\ii y$.  Without loss of generality, let us assume that 
		\begin{equation}
			\frac{\partial \Re \tilde{\eta}_q}{\partial x}(z^*;\boldsymbol{t}^*) \neq 0.
		\end{equation}
		Now, think of the left hand side of \eqref{Main Eq} as a map
		\begin{equation}
			\Re \tilde{\eta}_q\left(x,\boldsymbol{w}\right) :  \R \times W \to \R,
		\end{equation}
		where $z=x+\ii y$, and an element $\boldsymbol{w} \in W \simeq \R^{4p-1}$ represents the variable $y$ and the real and imaginary parts of the parameters in the external field:
		\[ \boldsymbol{w} = \left( y , \Re t_1, \Im, t_1, \cdots, \Re t_{2p-1}, \Im t_{2p-1}  \right)^T.  \]
		Now, by the real-analytic Implicit Function Theorem \cite{KrantzParks}, we know that there exists a neighborhood $U$ of \[ \boldsymbol{w}^* \equiv \left( y^* , \Re t^*_1, \Im, t^*_1, \cdots, \Re t^*_{2p-1}, \Im t^*_{2p-1}  \right) \] and a  real-analytic map $\Psi : U \to \R$, with 
		\[ \Psi\left( \boldsymbol{w}^*  \right) = x^* \]
		and \[ 	\Re \tilde{\eta}_q(\Psi(\boldsymbol{w}), \boldsymbol{w}) = 0. \]
		That is to say that for any $y$ in a small enough neighborhood of $y^*$ and for any $ \boldsymbol{t} \equiv \left(\Re t_1, \Im, t_1, \cdots, \Re t_{2p-1}, \Im t_{2p-1}  \right)^T$  in a small enough neighborhood of $\boldsymbol{t}^*$, there is an $x \equiv x(y,\boldsymbol{t})$ such that $z=x+iy$ lies on the critical graph $\mathscr{J}^{(q)}_{\boldsymbol{t}}$. The real-analyticity of $\Psi$, in particular, ensures that $\mathscr{J}^{(q)}_{\boldsymbol{t}}$ deforms continuously with respect to $\boldsymbol{t}$. This finishes the proof of Theorem \ref{continuous deformations of one cut critical graph1}.
% \end{proof}

\subsection{Proof of Theorem \ref{main thm}} Let us start with the following two lemmas.

\begin{lemma}\label{z_l cont}
	The points $z_j(\boldsymbol{t})$ ,  $j=1,\cdots,r$,  depend continuously on $\boldsymbol{t}$.
\end{lemma}
\begin{proof}
The right hand side of \eqref{h h h} clearly depends continuously on $\boldsymbol{t}$ (since the $\boldsymbol{t}$-dependence in $R$ is through the end-points which do depend continuously on $\boldsymbol{t}$). So the zeros of $h(z;\boldsymbol{t})$, being  $z_\ell(\boldsymbol{t})$ ,  $\ell=1,\cdots,r$,  depend continuously on $\boldsymbol{t}$.
	\end{proof}

\begin{lemma}\label{Lemma one or two vetex polygons are ruled out}
	There are no singular finite geodesic polygons with one or two vertices associated with the quadratic differential $Q(z)\dd z^2$ given by \eqref{em25}.
\end{lemma}

\begin{proof}
	The proof follows immediately from the Teichm{\"u}ller's lemma and the fact that $Q$ is a polynomial.
\end{proof}

Now we prove Theorem \ref{main thm}. Let  $\boldsymbol{t}^*$  be a regular $q$-cut point. We show that there exists a small enough neighborhood of $\boldsymbol{t}^*$ in which all the requirements of Definition \ref{Def q cut} hold simultaneously.  We prove this in the following two mutually exclusive cases:

\begin{itemize}
	\item[(a)]  when none of the points $z_\ell(\boldsymbol{t}^*)$ lie on $\mathscr{J}^{(q)}_{\boldsymbol{t}^*} \setminus J^{(q)}_{\boldsymbol{t}^*}$, and
	\item[(b)] when one or more of the points $z_\ell(\boldsymbol{t}^*)$ lie on $\mathscr{J}^{(q)}_{\boldsymbol{t}^*} \setminus J^{(q)}_{\boldsymbol{t}^*}$.  
\end{itemize}

Let us first consider the case (a) above. So we are at a regular $q$-cut point $\boldsymbol{t}^*$ where we know that

\begin{equation}\label{Aj neq 0}
	A_{\ell}(\boldsymbol{t}^*)\neq 0, \qquad \ell=1,\cdots,r,
\end{equation}
where
\begin{equation}\label{Al}
A_{\ell}(\boldsymbol{t}):=\Re \eta_q\left(z_{\ell}(\boldsymbol{t});\boldsymbol{t}\right).
\end{equation}

For $\ep>0$, let $D_{\ep}(\boldsymbol{t}^*)$ denote the open set of all points $\boldsymbol{t}$ such that 
\[ |\Re t_k -\Re t^*_{k}| < \ep, \qandq |\Im t_k -\Im t^*_{k}| < \ep, \qquad \mbox{for} \qquad k=1,\cdots,2p-1.  \]

Since the functions $A_\ell(\boldsymbol{t})$ are continuous functions of $\boldsymbol{t}$, for each $\ell=1,\cdots,r$ there exists $\ep_\ell>0$ such that for all $\boldsymbol{t} \in D_{\ep_\ell}(\boldsymbol{t}^*)$ the inequalities $A_{\ell}(\boldsymbol{t})\neq 0$ hold for each $\ell=1,\cdots,r$. Let $\ep:= \underset{1 \leq \ell \leq r}{\min} \ep_\ell$. The claim is that for all $\boldsymbol{t} \in \Om =: \Om_0 \cap D_{\ep}(\boldsymbol{t}^*)$ (see the proof of Theorem \ref{solvability thm} to recall the open set $\Om_0$) all requirements of Definition \ref{Def q cut} hold.  It is obvious that the second requirement of Definition \ref{Def q cut} holds by the choice of $\ep$. Suppose that condition (3) of Definition \ref{Def q cut} does not hold for some $\tilde{\boldsymbol{t}} \in \Om$. Let $K(\boldsymbol{t}^*)$ denote the infinite geodesic polygon which hosts the complementary contour $\Ga_{\boldsymbol{t}}(b_q(\boldsymbol{t}),+\infty)$ as required by condition (3) of Definition \ref{Def q cut}. Due to Theorem \ref{continuous deformations of one cut critical graph1} this is only possible if 

\begin{itemize}
	\item[(a-i)] one or more points on the boundaries $\ell^{(b_q)}_2$ and $\ell^{(b_q)}_3$ of the infinite geodesic polygon $K(\boldsymbol{t}^*)$ continuously deform (as $\boldsymbol{t}^*$ deforms to $\tilde{\boldsymbol{t}}$) to coalesce together and block the access of a complementary contour from $b_q$ to $+\infty$, or
	\item[(a-ii)] if there are one or more humps in $K$ (see Remark \ref{humps in K}), one or more points on the boundaries $\ell^{(b_q)}_2$ or $\ell^{(b_q)}_3$ of the infinite geodesic polygon $K(\boldsymbol{t}^*)$ continuously deform (as $\boldsymbol{t}^*$ deforms to $\tilde{\boldsymbol{t}}$) to coalesce with the hump(s) and block the access of a complementary contour from $b_q$ to $+\infty$. This case necessitates $p>q$ which ensures the existence of humps as parts of the critical graph.
\end{itemize}

Notice that if there are \textit{no humps} in $K$, in particular when $p=q$ or $p<q$ which means there are no humps at all, then the only possibility to block the access from $b_q$ to $+\infty$ is what mentioned above in case (a-i).
We observe that the case (a-i) above is actually impossible by Lemma \ref{Lemma one or two vetex polygons are ruled out} as it necessitates a geodesic polygon with two vertices. 

So we just investigate the case (a-ii). Consider a point of coalescence $\tilde{z}$. Notice that $\tilde{z}$ can not be $b_q$ itself because for all $\boldsymbol{t} \in \Om_0$ there are only three emanating critical trajectories from $b_q$. At such a point we would have four emanating local trajectories from $\tilde{z}$ (or a higher even number of emanating local trajectories from $\tilde{z}$ if more than just two points come together at $\tilde{z}$) which is an indication that $\tilde{z}$ is a critical point of the quadratic differential. This is a contradiction, since $z_\ell(\tilde{\boldsymbol{t}})$, $\ell=1,\cdots,r$, do not lie on the critical trajectories by the choice of $\ep$ and hence $\tilde{z} \neq z_{\ell}(\tilde{\boldsymbol{t}})$. Moreover the quadratic differential $Q(z)\dd z^2$ given by \eqref{em25} does not have any critical points other than $a_j(\tilde{\boldsymbol{t}}), b_j(\tilde{\boldsymbol{t}})$ and $z_\ell(\tilde{\boldsymbol{t}})$, $j=1,\cdots,q$, $\ell=1,\cdots,r$. This finishes the proof that condition (3) of Definition \ref{Def q cut} holds for all $\boldsymbol{t} \in \Om$. Similar arguments show that the conditions (4) and (5) of Definition \ref{Def q cut} must also hold for all $\boldsymbol{t} \in \Om$. 

Now it only remains to consider the first requirement of Definition \ref{Def q cut}. Assume, for the sake of arriving at a contradiction that for some $\tilde{\boldsymbol{t}} \in \Om$ there is at least one index $j_1=1,\cdots,q$ for which the first requirement fails. Notice that there could not be more than one connection by Lemma \ref{Lemma one or two vetex polygons are ruled out}. So the only possibility to consider is that there is \textit{no connection} between $a_{j_1}(\tilde{\boldsymbol{t}})$ and $b_{j_1}(\tilde{\boldsymbol{t}})$. So the three local trajectories emanating from $a_{j_1}(\tilde{\boldsymbol{t}})$ and $b_{j_1}(\tilde{\boldsymbol{t}})$ must end up at $\infty$ and can not encounter $z_\ell(\tilde{\boldsymbol{t}})$, $\ell=0,\cdots,r$ by the choice of $\ep$. However this is impossible since there are at least $4(q-1)+6$ rays emanating from the end points and approaching infinity. There are already $4(p-q)$ rays ending up at $\infty$ from the $2(p-q)$ existing humps. This in total gives at least $4(q-1)+6+4(p-q) = 4p+2$ directions at $\infty$. Recall that we can have only $4p$ solutions at $\infty$. This means we have at least $2$ more solutions at $\infty$ than what is allowed. This means that at least $2$ rays emanating from the endpoints must connect to one or more humps. But this means we have at least two extra critical points other than $a_j(\tilde{\boldsymbol{t}}), b_j(\tilde{\boldsymbol{t}})$ and $z_\ell(\tilde{\boldsymbol{t}})$, $j=1,\cdots,q$, $\ell=1,\cdots,r$, which is a contradiction. This finishes the proof that the first requirement of Definition \ref{Def q cut} holds for all  $\boldsymbol{t} \in \Om$. Therefore, for case (a) we have shown that all the requirements of Definition \ref{Def q cut} hold simultaneously.

Notice that the proof of case (b) above (when $\boldsymbol{t}^*$ is a regular $q-$cut point and one or more of the points $z_\ell(\boldsymbol{t}^*)$ lie on  $\mathscr{J}^{(q)}_{\boldsymbol{t}^*} \setminus J^{(q)}_{\boldsymbol{t}^*}$) is very similar. To that end, let $1 \leq m \leq r-1$ be such that for all indices $$\{\ell_1, \cdots, \ell_m\} \subset \{1,\cdots,r\}$$ the points $z_{\ell_k}(\boldsymbol{t}^*)$, $1\leq k\leq m$, \textit{do not} lie on $\mathscr{J}^{(q)}_{\boldsymbol{t}^*} \setminus J^{(q)}_{\boldsymbol{t}^*}$. For these indices we define $\ep_{\ell_k}$ as above using the functions $A_{\ell_k}$ in \eqref{Al}. 

Now let us consider the rest of the indices
$$\{\ell_{m+1}, \cdots, \ell_r\} \subset \{1,\cdots,r\}$$
for which the points $z_{\ell_j}(\boldsymbol{t}^*)$, $m+1 \leq j \leq r$, \textit{do} lie on $\mathscr{J}^{(q)}_{\boldsymbol{t}^*} \setminus J^{(q)}_{\boldsymbol{t}^*}$. We claim that there is an $\ep_{\ell_j}>0$, for each  $m+1 \leq j \leq r$, such that for all $\boldsymbol{t} \in D_{\ep_{\ell_j}}(\boldsymbol{t}^*)$, the point $z_{\ell_j}(\boldsymbol{t})$ does not lie on $J^{(q)}_{\boldsymbol{t}}$. Indeed, since for each $\boldsymbol{t}$ the set $J^{(q)}_{\boldsymbol{t}}$ is compact, the distance function \[ \mathrm{d}_j(\boldsymbol{t}) := \mathrm{dist} \left( z_{\ell_j}(\boldsymbol{t}), J^{(q)}_{\boldsymbol{t}} \right) \equiv \underset{z\in J^{(q)}_{\boldsymbol{t}}}{min} \{ |z_{\ell_j}(\boldsymbol{t}) - z| \} \]
is well-defined and is a continuous function of $\boldsymbol{t}$ due to Theorem \ref{continuous deformations of one cut critical graph1} and Lemma \ref{z_l cont}. For the $\boldsymbol{t}^*$ under consideration we know that $\mathrm{d}_j(\boldsymbol{t}^*)>0$, and by the continuity of $\mathrm{d}_j$, there is an $\ep_{\ell_j}$ such that for all $\boldsymbol{t}$ in an $\ep_{\ell_j}$-neighborhood of $\boldsymbol{t}^*$ we have $\mathrm{d}_j(\boldsymbol{t})>0$.

\begin{comment}
	For the sake of arriving at a contradiction suppose that there is no such $\ep_{\ell_j}$. This means that there is a sequence $\{\boldsymbol{t}_k\}^{\infty}_{k=1}$ converging to $\boldsymbol{t}_*$ such that for all $k$, $z_{\ell_j}(\boldsymbol{t}_k) \in J^{(q)}_{\boldsymbol{t}_k}$. However, due to Theorem \ref{continuous deformations of one cut critical graph} and Lemma \ref{z_l cont} this means that the sequence of points $z_{\ell_j}(\boldsymbol{t}_k)  \in J^{(q)}_{\boldsymbol{t}_k}$ converges to $z_{\ell_j}(\boldsymbol{t}^*) \in  \mathscr{J}^{(q)}_{\boldsymbol{t}^*} \setminus J^{(q)}_{\boldsymbol{t}^*}$. This means that we have a geodesic polygon with two vertices at $z_{\ell_j}(\boldsymbol{t}^*)$ and one of the endpoints $a_k(\boldsymbol{t}^*), b_k(\boldsymbol{t}^*)$ which is a contradiction with Lemma \ref{Lemma one or two vetex polygons are ruled out}. Therefore $\ep_{\ell_j}>0$,  $m+1 \leq j \leq r$ exist.
\end{comment}

Again let $\ep:= \underset{1 \leq \ell \leq r}{\min} \ep_\ell$. The claim then is that for all $\boldsymbol{t} \in \Om$  all requirements of Definition \ref{Def q cut} hold.  It is obvious that the second requirement of Definition \ref{Def q cut} holds by the choice of $\ep$, and if any other requirement of Definition \ref{Def q cut} does not hold  for some $\boldsymbol{t} \in \Om$ , analogous reasoning as provided in case (a) above shows that one gets a contradiction.

\section{Conclusion}

In this article we have provided a simple and yet self-contained proof of the openness of the regular $q$-cut regime when the external field is a complex polynomial of even degree. We have also proven that the solvability of the $q$-cut end-point equations persists in a small enough neighborhood of a regular $q$-cut point in the parameter space. In addition, we have shown that the real and imaginary parts of the endpoints are real analytic with respect to the real and imaginary parts of the parameters of the external field, and that the critical graph of the underlying quadratic differential depends continuously on $\boldsymbol{t}$.

As discussed in the introduction, we could have considered other classes of admissible contours different from the one associated with the real axis, for the even degree polynomial \eqref{int1}. Yet, multiple other cases would have arised if one started with an \textit{odd-degree} polynomial external field, then considered its classes of admissible sectors and contours, and finally solved the max-min variational problem for the collection of contours from that class\footnote{Notice that in the case where the degree of the polynomial external field is odd, the full real axis can not lie in any admissible sector, as the condition \eqref{growth condition} is not astisfied as $z$ approaches to $\infty$ along the negative real axis. See e.g. Figure 1 of \cite{BDY} for the cubic external field, and Figure 2 of \cite{KS} for a quintic one.}.   However, to that end, even though we have made the simplifying assumption on fixing the degree of external field to be even, and our fixed choice of admissible contours, we would like to emphasize that our arguments presented in this paper still work in the other cases as long as one considers a single curve going to infinity inside any two admissible sectors.

\begin{comment}
	
Due to this lemma the components of the stable lands
\begin{equation}\label{stable lands}
\mathcal{S}_{\boldsymbol{t}^*}:=\left\{ z : \Re \left[ \eta_q(z;\boldsymbol{t}^*) \right]<0 \right\},
\end{equation}
slightly deform when we change $\boldsymbol{t}$ in a small neighborhood of $\boldsymbol{t}^*$. We can indeed choose the neighborhood $\widehat{\Om}$ small enough, possibly smaller than the neighborhood $\Om_0$ in the statement of Theorem \ref{Thm solvability of enpoints is an open condition}, so that for each $\boldsymbol{t} \in \widehat{\Om}$ the components of $\mathcal{S}_{\boldsymbol{t}}$ can still contain complementary arcs as described in items (3) and (4) of Definition \ref{Def three cut sigma}. 
\end{comment}

\section*{Acknowledgements}
This material is based upon work supported by the National Science Foundation
under Grant No. DMS-1928930. We gratefully acknowledge the Mathematical Sciences Research Institute, Berkeley, California and the organizers of the semester-long program \textit{Universality and Integrability in Random Matrix Theory and Interacting Particle Systems} for their support in the Fall of 2021, during which P.B., R.G., and K.M. worked on this project. The work of M.B. was supported in part by the Natural Sciences and Engineering Research Council of Canada (NSERC) grant RGPIN-2016-06660. K.M. acknowledges support from the NSF grant DMS-1733967. A.T. acknowledges support from the NSF grant DMS-2009647.

\end{document}